\UseRawInputEncoding
\documentclass[a4paper]{article}
\pdfoutput=1
\usepackage[all]{xy}
\usepackage[latin1]{inputenc}        %accents
\usepackage[dvips]{graphics,graphicx}
\usepackage{amsfonts,amssymb,amsmath,color,mathrsfs, amstext}
\usepackage{amsbsy, amsopn, amscd, amsxtra, amsthm,authblk, enumerate}
\usepackage{enumerate,algorithmic,algorithm}
\usepackage{upref}
\usepackage[colorlinks,
            linkcolor=red,
            anchorcolor=red,
            citecolor=red
            ]{hyperref}

\usepackage{geometry}
\usepackage[displaymath]{lineno}
\usepackage{float}
\usepackage{yhmath}
\usepackage[normalem]{ulem}
\usepackage{subfig}

\numberwithin{equation}{section}

\def\cL{\mathcal{L}}

\def\cA{\mathcal{A}}
\def\cR{\mathcal{R}}

\newtheorem{theorem}{Theorem}[section]
\newtheorem{lemma}{Lemma}[section]

\newtheorem{proposition}{Proposition}[section]
\newtheorem{remark}{Remark}[section]
\newtheorem{corollary}{Corollary}[section]
\newtheorem{definition}{Definition}[section]

\begin{document}

\title{Quasi-random splitting method for accurate and efficient multiphysics simulation}
\author[a,b,c]{Lei Li\thanks{E-mail: leili2010@sjtu.edu.cn}}
\author[a]{Yunxiao Liu\thanks{E-mail: lyx112358@sjtu.edu.cn}}
\author[a]{Chenchen Wan\thanks{E-mail: wanchch7@sjtu.edu.cn}}
\affil[a]{School of Mathematical Sciences, Shanghai Jiao Tong University, Shanghai, 200240, P.R.China.}
% \affil[b]{Zhiyuan College, Shanghai Jiao Tong University, Shanghai, 200240, P.R.China.}
\affil[b]{Institute of Natural Sciences, MOE-LSC, Shanghai Jiao Tong University, Shanghai, 200240, P.R.China.}
\affil[c]{Shanghai Artificial Intelligence Laboratory}
\date{}
\maketitle

\begin{abstract}
We propose a quasi-random operator splitting method for evolution equations driven by multiple mechanisms. The method uses a one-dimensional low-discrepancy sequence to determine the ordering of the subflows through a direct point-to-permutation construction, while applying each subflow exactly once per time step. Thus, for a decomposition into \(p\) operators, the proposed method requires only \(p\) subflow evaluations per step, compared with essentially \(2p-2\) evaluations for the classical multi-operator Strang splitting when \(p\geq 3\). In contrast to randomized splitting, the quasi-random ordering is deterministic once the underlying sequence is fixed and avoids the run-to-run fluctuations caused by random permutations. To analyze the resulting scheme, we develop a convergence framework that connects the discrepancy of the induced operator orderings with cancellation of the leading commutator errors accumulated over time. For an arbitrary fixed number of bounded linear operators, we establish a global error of order \(O(\tau^2\log N)\), equivalently \(O(\tau^2|\log\tau|)\) for fixed terminal time, which is essentially second order. We further extend the same cancellation mechanism to the nonlinear Allen--Cahn equation and obtain the same essentially second-order convergence behavior under suitable Sobolev regularity assumptions. Numerical experiments for bounded linear systems and the Allen--Cahn equation confirm the predicted convergence rates and show that the proposed quasi-random splitting strategy retains the low per-step computational cost of a Lie-type splitting method while achieving accuracy close to second order.
\end{abstract}


\section{Introduction}

Operator splitting is a classical numerical strategy for evolution equations whose dynamics are driven by several distinct mechanisms. Its basic idea is to decompose a complicated system into simpler subproblems, solve these subproblems separately over each time step, and compose the resulting subflows to approximate the full evolution. This viewpoint is particularly attractive for multiphysics systems, where different terms often have different analytic and numerical structures. Typical examples include reaction--diffusion equations, advection--diffusion equations, phase-field models, and coupled systems arising in fluid dynamics, materials science, and computational physics \cite{strang1968construction,iserles2009first,glowinski2017splitting,anderson1998diffuse,yue2004diffuse,jacqmin1999calculation,praetorius2015navier}.

In abstract form, we consider
\begin{equation}\label{eq:intro:abstract}
\frac{\mathrm d}{\mathrm dt}u(t)
=
\sum_{j=1}^{p}\mathcal A_j\bigl(u(t)\bigr),
\qquad
u(0)=u_0,
\end{equation}
where $\mathcal A_1,\dots,\mathcal A_p$ are evolution operators, possibly nonlinear. If $S_j(t)$ denotes the flow of the $j$-th subproblem, then a splitting method replaces the exact evolution over one time step by a composition of the subflows $S_j$.

Among the most classical schemes are Lie splitting and Strang splitting \cite{strang1968construction,iserles2009first,glowinski2017splitting,mclachlan2002splitting,blanes2008splitting,Blanes_Casas_Murua_2024}. For a time step $\tau>0$, the Lie splitting uses
\begin{equation}\label{eq:intro:lie}
u^{n+1}
=
S_p(\tau)\cdots S_2(\tau)S_1(\tau)u^n.
\end{equation}
It applies each subflow once and is therefore inexpensive, but the resulting Lie composition is generally only first-order accurate, with the leading splitting error arising from non-commutativity of the split operators. A standard second-order alternative is the symmetric multi-operator Strang composition
\begin{equation}\label{eq:intro:strang}
u^{n+1}
=
S_1(\tau/2)\cdots S_{p-1}(\tau/2)
S_p(\tau)
S_{p-1}(\tau/2)\cdots S_1(\tau/2)u^n.
\end{equation}
For $p=2$, adjacent half steps can be combined efficiently. For $p\ge3$, however, the standard symmetric composition requires asymptotically $2p-2$ subflow applications per time step, whereas Lie splitting uses only $p$. This naturally raises the question of whether one can achieve a more favorable balance between computational cost and convergence order than is possible with the classical Lie and Strang splittings.

The ordering of the split operators plays a direct role in the splitting error. In the bounded linear setting, the Baker--Campbell--Hausdorff expansion and direct Taylor expansion show that the leading local error is expressed through commutators of the split generators \cite{iserles2009first,mclachlan2002splitting,hairer2006geometric}. Changing the relative order of two operators changes the sign of the corresponding commutator contribution. This observation suggests that varying the ordering from one time step to the next can produce cancellation in the accumulated local errors.

Randomized operator splitting provides one way to exploit this cancellation mechanism while keeping a low per-step computational cost
\cite{li2025convergence,li2025ergodicity,cho2024doubling}.
Instead of using the same Lie ordering at every time step, the operator ordering is randomized from step to step. A typical randomized step takes the form
\begin{equation}\label{eq:intro:random}
u^{n+1}
=
S_{\xi_p}(\tau)\cdots S_{\xi_1}(\tau)u^n,
\end{equation}
where $(\xi_1,\dots,\xi_p)$ is a random permutation of $\{1,\dots,p\}$.
Since each operator is still applied only once per step, the method retains the low computational cost of Lie splitting. At the same time, the random changes of ordering allow the leading commutator contributions with opposite signs to cancel statistically over successive time steps, rather than accumulating coherently as in a fixed Lie ordering. This temporal averaging can therefore improve the global convergence behavior; for example, an expected single-run convergence order of $3/2$ and a second-order bias have been established for the random splitting method in the Allen--Cahn setting \cite{li2025convergence}. Thus, randomization can improve the convergence behavior without increasing the number of subflow applications per time step. The price for this improvement is randomness: a single realization is subject to statistical fluctuations, and different runs generally produce different numerical trajectories.

To address the limitations associated with randomized splitting, we turn to quasi-random sequences, which are deterministic low-discrepancy sequences designed to achieve a highly uniform coverage of the sampling domain \cite{caflisch1998monte,dick2010digital,dick2013high,kuipers2012uniform}. The low-discrepancy property of these sequences suggests using them to distribute operator orderings more evenly across time steps. Unlike independently randomized orderings, the resulting ordering sequence is completely reproducible once the underlying low-discrepancy sequence is fixed, while still providing a well-distributed exploration of the possible operator orderings. Accordingly, we use a low-discrepancy sequence to determine the operator ordering at each time step. This preserves the low per-step cost of Lie splitting while enabling deterministic cancellation of the leading commutator errors across time steps.

To realize this strategy, we introduce a one-dimensional point-to-permutation construction that maps each low-discrepancy point to an ordering of the $p$ split operators. The permutation can be generated directly by a mixed-radix decoding procedure without explicitly enumerating all $p!$ possible orderings. Each time step still applies every subflow exactly once, so the method retains the low per-step cost of Lie splitting. The resulting deterministic ordering sequence provides the structure needed to control the accumulation of the leading splitting errors, as established in the convergence analysis below.

For bounded linear problems with an arbitrary fixed number $p$ of operators, we show that the leading local splitting error decomposes into pairwise commutator contributions whose signs are determined by the operator ordering. The low-discrepancy ordering yields logarithmic cancellation of these pairwise contributions across time steps. As a result, we establish the global error estimate
\[
\max_{0\le n\le N}\|u_n-u(t_n)\|
\le
C_p\tau^2(1+\log N)
=
O\!\left(\tau^2(1+|\log\tau|)\right),
\]
for every fixed $p$.

We further extend the analysis to a nonlinear three-operator Allen--Cahn equation with a background flow,
\begin{equation}\label{eq:intro:ac}
\partial_t u + v(x)\cdot\nabla u
=
\nu\Delta u-f(u),
\qquad
f(u)=u^3-u,
\end{equation}
posed on the torus with periodic boundary conditions. The equation is naturally decomposed into advection, diffusion, and reaction subproblems. The analysis is more delicate than in the bounded linear setting because the differential operators are unbounded and the reaction term is nonlinear. We establish an abstract convergence result for three-operator splitting and verify its assumptions for this problem under suitable Sobolev regularity conditions. This yields the same essentially second-order error estimate,
\[
\max_{0\le n\le N}\|u_n-u(t_n)\|_{k,r}
\le
C\tau^2(1+\log N).
\]

The main contributions of this paper are summarized as follows. First, we propose a quasi-random splitting method in which a one-dimensional low-discrepancy sequence determines the operator ordering at each time step through a direct point-to-permutation construction. Each of the $p$ subflows is applied exactly once per step, so the method retains the low per-step cost of Lie splitting. Second, we establish logarithmic cancellation for the relative ordering of every pair of operators, providing the key mechanism for controlling the accumulation of the leading splitting errors. Third, for an arbitrary fixed number $p$ of bounded linear operators, we prove the essentially second-order global error estimate $O(\tau^2(1+\log N))$. Fourth, we establish an abstract convergence theorem for three-operator splitting and apply it to the nonlinear Allen--Cahn equation with a background flow, obtaining the same essentially second-order convergence rate under suitable Sobolev regularity assumptions. Finally, numerical experiments for bounded linear systems and Allen--Cahn-type equations confirm the predicted near-second-order behavior and illustrate the favorable balance between accuracy and per-step computational cost.

The rest of the paper is organized as follows. Section~\ref{sec:setup} introduces the abstract splitting formulation and the point-to-permutation construction driven by a one-dimensional low-discrepancy sequence. Section~\ref{section:prelim on the quasi-random} develops the discrepancy estimates and the pairwise cancellation properties of the induced operator orderings. Section~\ref{sec:convergence} contains the convergence analysis: Section~\ref{sec:ode} treats an arbitrary fixed number of bounded linear operators, Section~\ref{sec:general_convergence} establishes an abstract convergence theorem for three-operator splitting, and Section~\ref{sec:pde} verifies its assumptions for the three-operator Allen--Cahn equation with a background flow. Section~\ref{sec:numerics} presents the numerical experiments, and Section~\ref{sec:conclusion} concludes the paper.

\section{Setup and Notations}\label{sec:setup}

\subsection{General setup of the quasi-random splitting method}

Let $X$ be a Banach space, and let $p\ge 2$. We consider the abstract
evolution problem
\begin{equation}\label{eq:setup:abstract}
\frac{\mathrm d}{\mathrm dt}u(t)
=
\sum_{j=1}^{p}\mathcal A_j(u(t)),
\qquad
u(0)=u_0\in X,
\end{equation}
where $\mathcal A_1,\dots,\mathcal A_p$ are given evolution operators
on $X$.

We denote by $T(t)$ the exact solution operator associated with
\eqref{eq:setup:abstract}, so that
\[
u(t)=T(t)[u_0].
\]
For each $j=1,\dots,p$, let $S_j(t)$ denote the flow generated by the
$j$-th subproblem
\begin{equation}\label{eq:setup:subproblem}
\frac{\mathrm d}{\mathrm dt}w(t)
=
\mathcal A_j(w(t)),
\qquad
w(0)=w_0.
\end{equation}
Thus,
\[
w(t)=S_j(t)[w_0].
\]

Let $\mathfrak S_p$ denote the set of all permutations of
$\{1,\dots,p\}$. For
\[
\pi=(\pi_1,\dots,\pi_p)\in\mathfrak S_p,
\]
we define the corresponding Lie-type splitting operator by
\begin{equation}\label{eq:setup:split_operator}
S^\pi(\tau)
:=
S_{\pi_p}(\tau)\cdots
S_{\pi_2}(\tau)S_{\pi_1}(\tau).
\end{equation}
Thus, the actual order in which the subflows are applied is
$\pi_1,\pi_2,\dots,\pi_p$.

Fix a terminal time $T>0$ and an integer $N\ge1$, and define the
uniform time grid by
\begin{equation}\label{eq:setup:grid}
\tau:=\frac{T}{N},
\qquad
t_n:=n\tau,
\qquad
n=0,1,\dots,N.
\end{equation}
In particular, $t_N=T$ exactly.

\subsection{Point-to-permutation ordering for the multi-operator method}
\label{sec:setup:interval_permutation}

We now define the quasi-random ordering mechanism used throughout the
multi-operator analysis. Given a point $z\in[0,1)$, the construction maps
$z$ directly to a permutation of the $p$ split operators through a
mixed-radix decoding procedure.

We first define the direct point-to-permutation map. Fix $z\in[0,1)$ and set
\[
x_1:=z,
\qquad
\mathcal I_1:=\{1,\dots,p\}.
\]
For $m=1,\dots,p-1$, let
\[
q_m:=p-m+1
\]
be the number of remaining operator indices and define
\begin{equation}\label{eq:setup:lehmer_digit}
c_m:=\lfloor q_mx_m\rfloor,
\qquad
x_{m+1}:=q_mx_m-c_m.
\end{equation}
Then
\[
0\le c_m\le q_m-1,
\qquad
x_{m+1}\in[0,1).
\]

Let $\pi_m$ be the $(c_m+1)$-st smallest element of the current ordered
set $\mathcal I_m$, and set
\[
\mathcal I_{m+1}
:=
\mathcal I_m\setminus\{\pi_m\}.
\]
After $p-1$ selections, let $\pi_p$ be the unique remaining element of
$\mathcal I_p$. We define
\begin{equation}\label{eq:setup:Pi_map}
\Pi_p(z)
:=
(\pi_1,\dots,\pi_p)
\in\mathfrak S_p.
\end{equation}

We next use this map to define the quasi-random splitting scheme. Let
\[
\{z_n\}_{n\ge1}\subset[0,1)
\]
be a one-dimensional low-discrepancy sequence. At time step
$n=0,\dots,N-1$, we define
\begin{equation}\label{eq:setup:interval_permutation_rule}
\sigma^n
:=
\Pi_p(z_{n+1}).
\end{equation}
The quasi-random splitting scheme is
\begin{equation}\label{eq:setup:interval_scheme}
u_{n+1}
=
S^{\sigma^n}(\tau)[u_n]
=
S_{\sigma_p^n}(\tau)\cdots
S_{\sigma_1^n}(\tau)u_n,
\qquad
n=0,\dots,N-1.
\end{equation}
Thus, the actual order in which the subflows are applied at time step
$n$ is
\[
\sigma_1^n,\sigma_2^n,\dots,\sigma_p^n,
\]
and every subflow is applied exactly once.

The complete procedure is summarized in Algorithm~\ref{Alg:qmc}.

\begin{algorithm}[H]
\caption{Quasi-random splitting with direct permutation decoding}
\label{Alg:qmc}
\small

\textbf{Require:}
Terminal time $T>0$, number of time steps $N$, initial value $u_0$,
number of operators $p$, and low-discrepancy points
$\{z_n\}_{n\ge1}\subset[0,1)$.

\begin{enumerate}
\setlength{\itemsep}{2pt}
\setlength{\parskip}{0pt}

\item Set
\[
\tau=\frac{T}{N}.
\]

\item For $n=0,1,\dots,N-1$:

\begin{enumerate}
\setlength{\itemsep}{2pt}
\setlength{\parskip}{0pt}

\item Set
\[
x\leftarrow z_{n+1},
\qquad
\mathcal I\leftarrow\{1,2,\dots,p\},
\]
where $\mathcal I$ is the ordered set of currently available operator
indices.

\item For $m=1,\dots,p-1$:

\begin{enumerate}
\setlength{\itemsep}{1pt}
\setlength{\parskip}{0pt}

\item Set
\[
q\leftarrow p-m+1,
\qquad
c\leftarrow\lfloor qx\rfloor,
\qquad
x\leftarrow qx-c.
\]

\item Let $\sigma_m^n$ be the $(c+1)$-st smallest element of
$\mathcal I$.

\item Remove the selected index:
\[
\mathcal I
\leftarrow
\mathcal I\setminus\{\sigma_m^n\}.
\]

\end{enumerate}

\item Let $\sigma_p^n$ be the unique remaining element of $\mathcal I$.

\item Advance the numerical solution by
\[
u_{n+1}
=
S_{\sigma_p^n}(\tau)\cdots
S_{\sigma_1^n}(\tau)u_n.
\]

\end{enumerate}

\item End for.

\end{enumerate}

\end{algorithm}

For an efficient implementation, the ordered active set $\mathcal I$ can be
maintained by a Fenwick tree. Each order-statistic selection and deletion then
requires $O(\log p)$ operations. Since $p-1$ such selections are performed at
each time step, the permutation-decoding cost is $O(p\log p)$ and the memory
cost is $O(p)$. In particular, the algorithm never enumerates or stores the
$p!$ possible permutations.

For the subsequent discrepancy analysis, it is useful to relate the direct
mixed-radix decoder to an equivalent interval representation. From
\eqref{eq:setup:lehmer_digit},
\[
z
=
\frac{c_1}{p}
+
\frac{c_2}{p(p-1)}
+\cdots+
\frac{c_{p-1}}{p!}
+
\frac{x_p}{p!},
\]
and hence
\[
p!z
=
\sum_{m=1}^{p-1}c_m(p-m)!
+x_p,
\qquad
0\le x_p<1.
\]
Therefore
\begin{equation}\label{eq:setup:permutation_rank}
k(z)
:=
\sum_{m=1}^{p-1}c_m(p-m)!
=
\lfloor p!z\rfloor.
\end{equation}

The digits $c_1,\dots,c_{p-1}$ are precisely the Lehmer digits of the
lexicographic rank $k(z)$. If $\operatorname{Unrank}_p(k)$ denotes the
permutation with lexicographic rank $k$, then
\begin{equation}\label{eq:setup:unrank}
\Pi_p(z)
=
\operatorname{Unrank}_p\!\left(\lfloor p!z\rfloor\right).
\end{equation}

Let
\[
\pi^{(r)}
:=
\operatorname{Unrank}_p(r-1),
\qquad
r=1,\dots,p!,
\]
and define
\begin{equation}\label{eq:setup:interval_partition}
J_r
:=
\left[
\frac{r-1}{p!},
\frac{r}{p!}
\right).
\end{equation}
Then
\begin{equation}\label{eq:setup:interval_to_permutation}
z\in J_r
\quad\Longleftrightarrow\quad
\Pi_p(z)=\pi^{(r)}.
\end{equation}

Thus, although the implementation does not enumerate the $p!$
permutations, the point-to-permutation map is mathematically equivalent
to assigning one permutation to each of the $p!$ equal subintervals of
$[0,1)$. This representation will be used in
Section~\ref{section:prelim on the quasi-random} to analyze the induced
pairwise operator orderings.

\subsection{Other notations}

To avoid confusion with the number $p$ of split operators, we use $r$ for the Lebesgue exponent in the PDE analysis. When the state space is a function space over a domain $\Omega$, we use the standard spaces $L^r(\Omega)$ and $W^{k,r}(\Omega)$.

For $1\le r<\infty$, define
\begin{equation}\label{eq:setup:Lp}
\|u\|_r
=
\left(\int_{\Omega}|u|^r\,dx\right)^{1/r}.
\end{equation}
For $r=\infty$, define
\begin{equation}\label{eq:setup:Linf}
\|u\|_{\infty}
=
\operatorname*{ess\,sup}_{x\in\Omega}|u(x)|.
\end{equation}

A multi-index is a tuple $\alpha=(\alpha_1,\dots,\alpha_d)\in\mathbb N^d$, where $\mathbb N=\{0,1,2,\dots\}$. We write
\[
|\alpha|=\alpha_1+\cdots+\alpha_d,
\qquad
D^\alpha u=\prod_{i=1}^{d}\left(\frac{\partial}{\partial x_i}\right)^{\alpha_i}u.
\]
For $k\in\mathbb N$ and $1\le r<\infty$, define
\begin{equation}\label{eq:setup:Sobolev}
\|u\|_{k,r}
=
\left(\sum_{|\alpha|\le k}\|D^\alpha u\|_r^r\right)^{1/r}.
\end{equation}
For $r=\infty$, set
\begin{equation}\label{eq:setup:Sobolev_infty}
\|u\|_{k,\infty}
=
\max_{|\alpha|\le k}\|D^\alpha u\|_\infty.
\end{equation}
Clearly, $L^r(\Omega)=W^{0,r}(\Omega)$ and $\|u\|_r=\|u\|_{0,r}$.

\section{Preliminaries on quasi-random sequences and ordering cancellation}
\label{section:prelim on the quasi-random}

Having introduced the point-to-permutation construction, we now develop the
quasi-random estimates needed for the convergence analysis. We first recall
the discrepancy bound for the one-dimensional radical-inverse sequence and
the Koksma inequality. These results yield logarithmic cancellation estimates
for a basic sign sequence and their weighted counterparts. We then extend the
same argument to the relative ordering of each pair of operators generated by
the point-to-permutation map, which gives the pairwise cancellation estimate
used in the subsequent convergence analysis.

\subsection{Preliminary results on the quasi-random sequence}

We first recall the notion of one-dimensional discrepancy.

\begin{definition}[One-dimensional discrepancy]\label{def:discrepancy}
Let $x_1,\dots,x_N\in[0,1)$ be given.
The discrepancy of the point set $\{x_n\}_{n=1}^N$ is defined by
\begin{equation}\label{eq:prelim:discrepancy}
D_N^*(x_1,\dots,x_N)
:=
\sup_{A\in[0,1]}
\left|
\frac{S_N(A)}{N}-A
\right|,
\end{equation}
where
\[
S_N(A):=\#\{1\le n\le N:\ x_n<A\}
\]
denotes the number of points contained in $[0,A)$.
\end{definition}

A standard one-dimensional quasi-random sequence is the radical-inverse
sequence in base $R$,
\begin{equation}\label{eq:prelim:qmc_sequence}
z_n:=\varphi_R(n),
\qquad
n=1,2,\dots,
\end{equation}
where $\varphi_R$ is defined as follows.
Let $R\ge2$ be an integer. For each integer $n\ge1$, write its radix-$R$
expansion as
\begin{equation}\label{eq:prelim:radix_expansion}
n
=
n_0+n_1R+\cdots+n_{M_n}R^{M_n},
\qquad
M_n:=\lfloor \log_R n\rfloor,
\end{equation}
where $n_j\in\{0,1,\dots,R-1\}$. The corresponding radical-inverse
function is given by
\begin{equation}\label{eq:prelim:radical_inverse}
\varphi_R(n)
:=
n_0R^{-1}
+n_1R^{-2}
+\cdots+
n_{M_n}R^{-M_n-1}.
\end{equation}

The following discrepancy estimate gives the basic quantitative property of
this sequence.

\begin{theorem}\label{thm:discrepancy_bound}
Let $\{z_n\}_{n\ge1}$ be the radical-inverse sequence defined above.
Then, for every $N\ge2$,
\begin{equation}\label{eq:prelim:discrepancy_bound}
D_N^*(z_1,\dots,z_N)
\le
\frac{3R-2}{\log R}\,
\frac{\log N}{N}.
\end{equation}
\end{theorem}

The proof can be found in \cite{halton1960efficiency}.

We will also use the following Koksma inequality.

\begin{theorem}[Koksma inequality]\label{thm:koksma}
Let $f:[0,1]\to\mathbb{R}$ be of bounded variation, and let
$x_1,\dots,x_N\in[0,1]$ be given. Then
\begin{equation}\label{eq:prelim:koksma}
\left|
\frac{1}{N}\sum_{n=1}^{N}f(x_n)
-
\int_0^1 f(t)\,\mathrm dt
\right|
\le
V(f)\,D_N^*(x_1,\dots,x_N),
\end{equation}
where $V(f)$ denotes the total variation of $f$ on $[0,1]$.
\end{theorem}

The proof can be found in \cite{kuipers2012uniform}. For completeness, we
include a brief proof. Without loss of generality, reorder the points so that
\[
0\le x_1\le x_2\le\cdots\le x_N\le1,
\]
and set $x_0:=0$ and $x_{N+1}:=1$. By integration by parts,
\begin{equation*}
\left|
\frac{1}{N}\sum_{n=1}^{N}f(x_n)
-
\int_0^1 f(t)\,\mathrm dt
\right|
=
\left|
\sum_{n=0}^{N}
\int_{x_n}^{x_{n+1}}
\left(t-\frac{n}{N}\right)\,\mathrm df(t)
\right|
\le
\sum_{n=0}^{N}
\int_{x_n}^{x_{n+1}}
\left|t-\frac{n}{N}\right|\,|\mathrm df(t)|.
\end{equation*}
We now recall the standard characterization of the one-dimensional
discrepancy,
\begin{equation}\label{eq:prelim:discrepancy_formula}
D_N^*(x_1,\dots,x_N)
=
\max_{1\le i\le N}
\max\left\{
\left|x_i-\frac{i}{N}\right|,
\left|x_i-\frac{i-1}{N}\right|
\right\}.
\end{equation}
Hence, for every $t\in[x_n,x_{n+1}]$,
\[
\left|t-\frac{n}{N}\right|
\le
D_N^*(x_1,\dots,x_N),
\qquad
n=0,1,\dots,N.
\]
Substituting this bound into the previous estimate gives
\[
\left|
\frac{1}{N}\sum_{n=1}^{N}f(x_n)
-
\int_0^1 f(t)\,\mathrm dt
\right|
\le
D_N^*(x_1,\dots,x_N)
\sum_{n=0}^{N}
\int_{x_n}^{x_{n+1}}|\mathrm df(t)|
=
V(f)\,D_N^*(x_1,\dots,x_N),
\]
which proves the inequality.

\subsection{The induced sign sequence and its cancellation}

We now introduce the sign sequence obtained by thresholding the quasi-random sequence at $1/2$.

\begin{definition}[Sign sequence]\label{def:sign_sequence}
Let $\{z_n\}_{n\ge1}$ be the quasi-random sequence defined in \eqref{eq:prelim:qmc_sequence}. 
We define the sign sequence $\{\xi_n\}_{n\ge1}$ by $\xi_n=1$ if $z_n\ge 1/2$ while $\xi_n=-1$ if $z_n<1/2$. Then, define the partial sums
\begin{equation}\label{eq:prelim:partial_sums}
S_n:=\sum_{k=1}^{n}\xi_k,
\qquad n\ge1,
\end{equation}
with the convention $S_0:=0$.
\end{definition}

With the above result, it is straightforward to obtain the bound for partial sums.
\begin{lemma}\label{lmm:partial_sum_bound}
One has
\begin{equation}\label{eq:prelim:partial_sum_bound}
|S_N|\le \frac{2(3R-2)}{\log R} \log N,
\qquad N\ge2.
\end{equation}
\end{lemma}

\begin{proof}
Consider the step function
\begin{equation}\label{eq:prelim:g_definition}
g(x)=
\begin{cases}
+1, & x\ge \dfrac12,\\[1mm]
-1, & x<\dfrac12.
\end{cases}
\end{equation}
Then, by Definition~\ref{def:sign_sequence},
$S_N=\sum_{k=1}^{N} g(z_k)$.
Theorem~\ref{thm:koksma} yields
\[
\frac{|S_N|}{N}
=
\left|
\frac{1}{N}\sum_{k=1}^{N} g(z_k)-\int_0^1 g(x)\,\mathrm dx
\right|
\le
2\,D_N^*(z_1,\dots,z_N).
\]
Applying the discrepancy estimate in Theorem~\ref{thm:discrepancy_bound}, we obtain
\[
|S_N|
\le
2N\,D_N^*(z_1,\dots,z_N)
\le
\frac{2(3R-2)}{\log R}\log N,
\]
which completes the proof.
\end{proof}

Applying Lemma~\ref{lmm:partial_sum_bound} for $2\le n\le N$ gives
\[
|S_n|
\le
\frac{2(3R-2)}{\log R}\log n
\le
\frac{2(3R-2)}{\log R}\log N.
\]
The cases $n=0$ and $n=1$ follow directly from $S_0=0$ and $|S_1|=1$;
the latter is also covered by the same bound for $N\ge2$. Hence we obtain
the following uniform estimate.
\begin{corollary}\label{cor:partial_sum_sup}
For every $N\ge2$,
\begin{equation}\label{eq:prelim:partial_sum_sup}
\max_{0\le n\le N}|S_n|
\le \frac{2(3R-2)}{\log R} \log N.
\end{equation}
\end{corollary}

The next theorem records a weighted version of the cancellation estimate; the same Abel-summation mechanism will reappear in the convergence analysis.

\begin{theorem}\label{thm:weighted_sum}
Let $f\in \mathrm{Lip}([0,T])\cap L^\infty(0,T)$, $\tau=T/N$ be the time step and $t_k=k\tau$. Let $\{\xi_k\}_{k\ge1}$ be the sign sequence defined in Definition~\ref{def:sign_sequence}. 
Then, it holds that
\begin{equation}\label{eq:prelim:weighted_sum}
\left|
\frac{1}{N}\sum_{k=1}^{N} f(t_k)\xi_k
\right|
\le
C_3\,\tau\log N,
\end{equation}
where $C_3$ can be taken as
\begin{equation}\label{eq:prelim:C3}
C_3 = \frac{2(3R-2)}{\log R} \left(
\mathrm{Lip}(f)+\frac{\|f\|_\infty}{T}
\right).
\end{equation}
\end{theorem}

\begin{proof}
By Abel's summation formula,
\begin{align}
\sum_{k=1}^{N} f(t_k)\xi_k
&=
f(t_N)S_N-\sum_{k=1}^{N-1} S_k\bigl(f(t_{k+1})-f(t_k)\bigr).
\label{eq:prelim:abel}
\end{align}
Taking absolute values and using Corollary~\ref{cor:partial_sum_sup}, we obtain
\begin{align}
\left|
\sum_{k=1}^{N} f(t_k)\xi_k
\right|
&\le
|f(t_N)|\,|S_N|
+
\sum_{k=1}^{N-1}|S_k|\,|f(t_{k+1})-f(t_k)|\notag\\
&\le
\frac{2(3R-2)}{\log R} \log N
\left(
\|f\|_\infty
+
\sum_{k=1}^{N-1}|f(t_{k+1})-f(t_k)|
\right).
\label{eq:prelim:weighted_sum_step1}
\end{align}
Since $f$ is Lipschitz continuous on $[0,T]$,
\[
|f(t_{k+1})-f(t_k)|
\le
\mathrm{Lip}(f)\,|t_{k+1}-t_k|
=
\mathrm{Lip}(f)\,\tau.
\]
Hence
\begin{equation}\label{eq:prelim:variation_on_grid}
\sum_{k=1}^{N-1}|f(t_{k+1})-f(t_k)|
\le
(N-1)\mathrm{Lip}(f)\tau
\le
T\,\mathrm{Lip}(f).
\end{equation}
Substituting \eqref{eq:prelim:variation_on_grid} into \eqref{eq:prelim:weighted_sum_step1} gives
\[
\left|
\sum_{k=1}^{N} f(t_k)\xi_k
\right|
\le \frac{2(3R-2)}{\log R}\bigl(\|f\|_\infty+T\,\mathrm{Lip}(f)\bigr)\log N.
\]
Dividing by $N$ and using $\tau=T/N$, one concludes the claim.

\end{proof}

The Abel-summation argument above is the form that will be used directly in
the convergence proofs. It is also useful to give a complementary geometric
interpretation of the same cancellation mechanism. The positive and negative
signs can be viewed as two collections of points in time. The partial-sum
bound implies that, apart from $O(\log N)$ unmatched points, the positive and
negative points can be matched in chronological order with a temporal
separation of order $O(\tau\log N)$. This matching provides a Wasserstein
interpretation of the cancellation quantified in
Theorem~\ref{thm:weighted_sum}.

Figure~\ref{fig:sign_matching} illustrates this chronological matching.
Writing
\[
I_+=\{p_1<p_2<\cdots\},
\qquad
I_-=\{m_1<m_2<\cdots\},
\]
the $i$-th positive point $t_{p_i}$ is paired with the $i$-th negative point
$t_{m_i}$. The relative temporal order of a matched pair is not fixed:
either $t_{p_i}<t_{m_i}$ or $t_{m_i}<t_{p_i}$ may occur. The connecting
segments therefore represent an undirected matching rather than a transport
in a prescribed time direction.

\begin{figure}[htbp]
\centering
\includegraphics[width=0.92\textwidth]{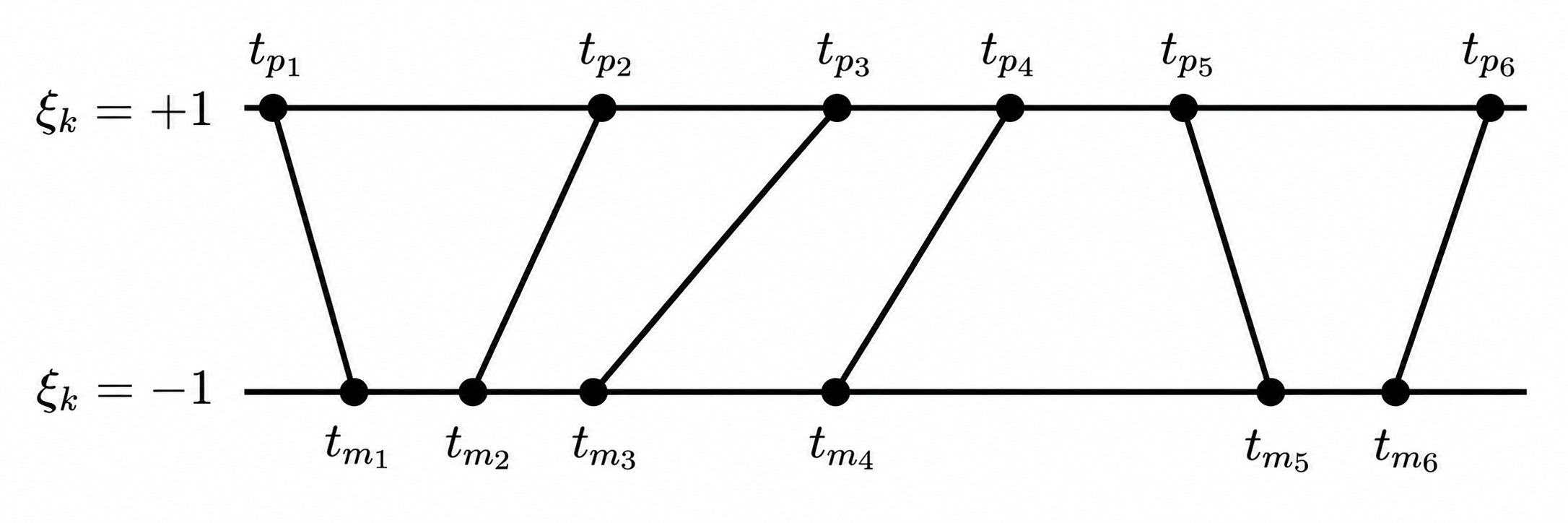}
\caption{Geometric illustration of the chronological matching between
positive and negative signs. The upper points $t_{p_i}$ correspond to
$\xi_k=+1$ and the lower points $t_{m_i}$ correspond to $\xi_k=-1$.
Each segment connects the $i$-th positive point with the $i$-th negative
point. The matched points may occur in either temporal order, and the
matching satisfies
$|t_{p_i}-t_{m_i}|\le 2C_R\tau\log N$.
Residual unmatched points are not shown.}
\label{fig:sign_matching}
\end{figure}

The following theorem makes this geometric interpretation precise.

\begin{theorem}[Geometric form of the sign cancellation]
\label{thm:measure_decomposition}
Let
\begin{equation}\label{eq:measure_decomp_nu}
\nu:=\sum_{n=1}^{N}\xi_n\delta_{t_n},
\qquad
t_n=n\tau,
\qquad
\tau=\frac{T}{N},
\end{equation}
be the signed counting measure associated with the sign sequence
$\{\xi_n\}_{n=1}^{N}$ from Definition~\ref{def:sign_sequence}. Define
\[
I_+:=\{n\in\{1,\dots,N\}:\xi_n=+1\},
\qquad
I_-:=\{n\in\{1,\dots,N\}:\xi_n=-1\},
\]
and let
\[
N_+:=\# I_+,
\qquad
N_-:=\# I_-,
\qquad
M:=\min\{N_+,N_-\}.
\]
Assume $M\ge1$. In particular, this condition holds for all sufficiently large
$N$ by Corollary~\ref{cor:partial_sum_sup}. Then one can decompose
\begin{equation}\label{eq:measure_decomp_main}
\nu=M(\mu^+-\mu^-)+\nu_r,
\end{equation}
where $\mu^\pm$ are probability measures on $[0,T]$ and $\nu_r$ is a
signed counting measure. If
\begin{equation}\label{eq:measure_decomp_CR}
C_R:=\frac{2(3R-2)}{\log R},
\end{equation}
then
\begin{equation}\label{eq:measure_decomp_tv}
\|\nu_r\|_{TV}
\le
C_R\log N,
\end{equation}
and, for every Wasserstein exponent $q\ge1$,
\begin{equation}\label{eq:measure_decomp_wq}
W_q(\mu^+,\mu^-)
\le
2C_R\,\tau\log N.
\end{equation}
\end{theorem}

\begin{proof}
We first consider the case $N_+\ge N_-$. Write the positive and negative
indices in increasing order as
\[
I_+=\{p_1<\cdots<p_{N_+}\},
\qquad
I_-=\{m_1<\cdots<m_{N_-}\}.
\]
Then $M=N_-$, and we define
\begin{equation}\label{eq:measure_decomp_mu_def}
\mu^+
:=
\frac1M\sum_{i=1}^{M}\delta_{t_{p_i}},
\qquad
\mu^-
:=
\frac1M\sum_{i=1}^{M}\delta_{t_{m_i}},
\qquad
\nu_r
:=
\sum_{i=M+1}^{N_+}\delta_{t_{p_i}}.
\end{equation}
The measures $\mu^\pm$ are probability measures, and
\eqref{eq:measure_decomp_main} follows immediately. If $N_+<N_-$, the same
construction is used after interchanging the roles of the positive and
negative indices; in that case the residual measure carries a negative sign.

The residual measure contains exactly $|N_+-N_-|$ unmatched Dirac masses.
Since
\[
S_N=N_+-N_-,
\]
Corollary~\ref{cor:partial_sum_sup} gives
\[
\|\nu_r\|_{TV}
=
|N_+-N_-|
=
|S_N|
\le
C_R\log N,
\]
which proves \eqref{eq:measure_decomp_tv}.

It remains to control the distance between the matched positive and negative
points. Set
\begin{equation}\label{eq:measure_decomp_HN}
H_N:=\max_{0\le n\le N}|S_n|.
\end{equation}
For each $i=1,\dots,M$, the fact that $p_i$ is the location of the $i$-th
positive sign gives
\[
S_{p_i}
=
i-(p_i-i)
=
2i-p_i.
\]
Similarly, because $m_i$ is the location of the $i$-th negative sign,
\[
S_{m_i}
=
(m_i-i)-i
=
m_i-2i.
\]
Consequently,
\begin{equation}\label{eq:measure_decomp_pair_distance_index}
m_i-p_i
=
S_{m_i}+S_{p_i},
\end{equation}
and hence
\begin{equation}\label{eq:measure_decomp_pair_distance_bound}
|m_i-p_i|
\le
2H_N.
\end{equation}
This identity gives a direct chronological matching estimate and avoids any
additional combinatorial pairing argument.

Consider the coupling
\[
\gamma
:=
\frac1M\sum_{i=1}^{M}
\delta_{(t_{p_i},t_{m_i})}
\]
between $\mu^+$ and $\mu^-$. For every $q\ge1$,
\begin{align*}
W_q(\mu^+,\mu^-)^q
&\le
\int_{[0,T]\times[0,T]}|x-y|^q\,\mathrm d\gamma(x,y)\\
&=
\frac1M\sum_{i=1}^{M}|t_{p_i}-t_{m_i}|^q.
\end{align*}
Since $t_n=n\tau$, \eqref{eq:measure_decomp_pair_distance_bound} yields
\[
|t_{p_i}-t_{m_i}|
=
\tau|p_i-m_i|
\le
2\tau H_N.
\]
Using Corollary~\ref{cor:partial_sum_sup},
\[
H_N\le C_R\log N,
\]
and therefore
\[
W_q(\mu^+,\mu^-)
\le
2\tau H_N
\le
2C_R\,\tau\log N.
\]
This proves \eqref{eq:measure_decomp_wq}.
\end{proof}

Theorem~\ref{thm:measure_decomposition} provides a geometric counterpart to
the cancellation quantified analytically in Theorem~\ref{thm:weighted_sum}.
Indeed, in the case $N_+\ge N_-$, the decomposition
\eqref{eq:measure_decomp_main} gives
\[
\sum_{n=1}^{N}f(t_n)\xi_n
=
M\left(
\int_0^T f\,\mathrm d\mu^+
-
\int_0^T f\,\mathrm d\mu^-
\right)
+
\int_0^T f\,\mathrm d\nu_r.
\]
Using the chronological coupling constructed above,
\[
\left|
\int f\,\mathrm d\mu^+
-
\int f\,\mathrm d\mu^-
\right|
\le
\operatorname{Lip}(f)\,
W_1(\mu^+,\mu^-)
\le
2C_R\operatorname{Lip}(f)\tau\log N.
\]
Since $M\le N/2$ and $N\tau=T$, while
\[
\left|
\int f\,\mathrm d\nu_r
\right|
\le
\|f\|_\infty\|\nu_r\|_{TV},
\]
we obtain
\[
\left|
\sum_{n=1}^{N}f(t_n)\xi_n
\right|
\le
C_R
\bigl(
T\,\operatorname{Lip}(f)+\|f\|_\infty
\bigr)\log N,
\]
which is of the same order as the estimate in
Theorem~\ref{thm:weighted_sum}.

Thus, Abel summation provides the algebraic estimate used in the convergence
analysis, whereas the measure decomposition gives a geometric interpretation
of the same cancellation. More generally, the matching argument only
requires a logarithmic bound on the partial sums of a
$\{-1,+1\}$-valued sequence. The same geometric interpretation therefore
applies to the pairwise sign sequences considered below once their
corresponding logarithmic cancellation estimate has been established.

\subsection{Pairwise cancellation of the induced operator orderings}
\label{sec:prelim:pairwise_cancellation}

We now establish the logarithmic cancellation estimate needed for the
multi-operator bounded-linear analysis. By the equivalent interval
representation derived in Section~\ref{sec:setup:interval_permutation},
the relative order of any fixed pair of operators can be represented by
a one-dimensional step function of the underlying low-discrepancy point.

\begin{definition}[Pairwise sign]\label{def:pairwise_sign}
For a permutation
$\pi=(\pi_1,\dots,\pi_p)\in\mathfrak S_p$ and an index
$\ell\in\{1,\dots,p\}$, define $\operatorname{pos}_{\pi}(\ell)$ by
\[
\operatorname{pos}_{\pi}(\ell)=r
\quad\Longleftrightarrow\quad
\pi_r=\ell.
\]
For each fixed pair $1\le i<j\le p$, define
\begin{equation}\label{eq:prelim:pairwise_sign_perm}
\Xi^{ij}(\pi):=
\begin{cases}
+1, & \operatorname{pos}_{\pi}(j)<\operatorname{pos}_{\pi}(i),\\[1mm]
-1, & \operatorname{pos}_{\pi}(i)<\operatorname{pos}_{\pi}(j).
\end{cases}
\end{equation}
Thus, $\Xi^{ij}(\pi)$ records the relative order of the pair $(i,j)$
in the permutation $\pi$.
\end{definition}

For the quasi-random ordering rule
\eqref{eq:setup:interval_permutation_rule}, define
\begin{equation}\label{eq:prelim:pairwise_sign_sequence}
\xi_n^{ij}:=\Xi^{ij}(\sigma^n),
\qquad
n=0,1,\dots,N-1.
\end{equation}
Using the equivalent interval representation from
Section~\ref{sec:setup:interval_permutation}, for each
$1\le i<j\le p$ we introduce the step function
$h_{ij}:[0,1)\to\{-1,+1\}$ by
\begin{equation}\label{eq:prelim:hij_definition}
h_{ij}(z):=\Xi^{ij}(\pi^{(r)})
\quad\text{if } z\in J_r,
\qquad
r=1,\dots,p!.
\end{equation}
Then, by construction,
\begin{equation}\label{eq:prelim:xi_h_relation}
\xi_n^{ij}=h_{ij}(z_{n+1}),
\qquad
n=0,1,\dots,N-1.
\end{equation}

\begin{lemma}\label{lem:prelim:hij_properties}
For every $1\le i<j\le p$, the function $h_{ij}$ satisfies
\begin{equation}\label{eq:prelim:hij_mean_zero}
\int_0^1 h_{ij}(z)\,\mathrm dz=0,
\end{equation}
and has bounded variation with
\begin{equation}\label{eq:prelim:hij_variation}
V(h_{ij})\le 2(p!-1).
\end{equation}
\end{lemma}

\begin{proof}
Since all intervals $J_r$ have the same length $1/p!$, we have
\[
\int_0^1 h_{ij}(z)\,\mathrm dz
=
\frac1{p!}\sum_{r=1}^{p!}\Xi^{ij}(\pi^{(r)})
=
\frac1{p!}\sum_{\pi\in\mathfrak S_p}\Xi^{ij}(\pi).
\]
It remains to show that the last sum is zero. For any
$\pi\in\mathfrak S_p$, let $\widetilde\pi$ be the permutation obtained
from $\pi$ by exchanging the positions of $i$ and $j$ while leaving all
other entries unchanged. The map
$\pi\mapsto\widetilde\pi$ is an involution of $\mathfrak S_p$ without
fixed points, and
\[
\Xi^{ij}(\widetilde\pi)
=
-\Xi^{ij}(\pi).
\]
Hence the permutations can be paired into canceling pairs, and therefore
\[
\sum_{\pi\in\mathfrak S_p}\Xi^{ij}(\pi)=0.
\]
This proves \eqref{eq:prelim:hij_mean_zero}.

For \eqref{eq:prelim:hij_variation}, note that $h_{ij}$ is a
one-dimensional step function that is constant on each interval $J_r$.
It can jump only at the $p!-1$ interior endpoints of the partition, and
each jump has magnitude at most $2$. Thus
\[
V(h_{ij})\le 2(p!-1).
\]
\end{proof}

\begin{theorem}[Pairwise logarithmic cancellation]
\label{thm:prelim:pairwise_cancellation}
Let $\{z_n\}_{n\ge1}$ be the radical-inverse sequence in base $R$, and let
$\{\sigma^n\}_{n\ge0}$ be generated by the quasi-random ordering rule
\eqref{eq:setup:interval_permutation_rule}. Then, for every
$1\le i<j\le p$ and every $N\ge2$,
\begin{equation}\label{eq:prelim:pairwise_partial_sum_bound}
\max_{1\le m\le N}
\left|
\sum_{n=0}^{m-1}\xi_n^{ij}
\right|
\le
C_{p,R}\log N,
\end{equation}
where one may take
\begin{equation}\label{eq:prelim:Cpair}
C_{p,R}
:=
2(p!-1)\frac{3R-2}{\log R}.
\end{equation}
\end{theorem}

\begin{proof}
Fix $1\le i<j\le p$. For $2\le m\le N$, by
\eqref{eq:prelim:xi_h_relation},
\[
\sum_{n=0}^{m-1}\xi_n^{ij}
=
\sum_{n=1}^{m}h_{ij}(z_n).
\]
Using \eqref{eq:prelim:hij_mean_zero} and applying the one-dimensional
Koksma inequality in Theorem~\ref{thm:koksma}, we obtain
\begin{align*}
\left|
\frac1m
\sum_{n=0}^{m-1}\xi_n^{ij}
\right|
&=
\left|
\frac1m\sum_{n=1}^{m}h_{ij}(z_n)
-
\int_0^1h_{ij}(z)\,\mathrm dz
\right|\\
&\le
V(h_{ij})D_m^*(z_1,\dots,z_m).
\end{align*}
By Lemma~\ref{lem:prelim:hij_properties} and
Theorem~\ref{thm:discrepancy_bound},
\[
\left|
\frac1m
\sum_{n=0}^{m-1}\xi_n^{ij}
\right|
\le
2(p!-1)\frac{3R-2}{\log R}
\frac{\log m}{m}.
\]
Multiplying by $m$ gives
\[
\left|
\sum_{n=0}^{m-1}\xi_n^{ij}
\right|
\le
C_{p,R}\log m
\le
C_{p,R}\log N,
\qquad
2\le m\le N.
\]

For $m=1$,
\[
\left|\xi_0^{ij}\right|=1,
\]
which is also bounded by $C_{p,R}\log N$ for $N\ge2$. Taking the maximum
over $m=1,\dots,N$ completes the proof.
\end{proof}

\section{Convergence analysis of the quasi-random splitting method}
\label{sec:convergence}

We first establish the convergence result for an arbitrary fixed number
$p$ of bounded linear operators by combining the pairwise local defect
decomposition with the logarithmic cancellation estimate from
Section~\ref{sec:prelim:pairwise_cancellation}. We then establish an abstract convergence theorem for three-operator
splitting under suitable stability and regularity assumptions, and
apply it to the three-operator Allen--Cahn equation with a background
flow.

\subsection{Bounded linear problems with $p$ operators}
\label{sec:ode}
\label{sec:ode:p_operator}

We first consider the quasi-random splitting method for an arbitrary fixed
number $p$ of bounded linear operators. Specifically,
\begin{equation}\label{eq:ode:p_model}
\frac{\mathrm d}{\mathrm dt}u(t)=\left(\sum_{m=1}^{p}A_m\right)u(t),
\qquad
u(0)=u_0\in X,
\end{equation}
where $X$ is a Banach space and $A_1,\dots,A_p\in\mathcal L(X)$. Define
\begin{equation}\label{eq:ode:Asum}
A_\Sigma:=\sum_{m=1}^{p}A_m,
\qquad
T(t):=e^{tA_\Sigma},
\qquad
S_m(t):=e^{tA_m}.
\end{equation}
For $\sigma=(\sigma_1,\dots,\sigma_p)\in\mathfrak S_p$, the Lie-type splitting step is
\begin{equation}\label{eq:ode:p_split_operator}
S^\sigma(\tau)=e^{\tau A_{\sigma_p}}\cdots e^{\tau A_{\sigma_1}}.
\end{equation}

\begin{lemma}[Pairwise local defect for $p$ bounded operators]
\label{lem:ode:p_local_defect}
For every $\sigma\in\mathfrak S_p$, the following expansion holds in
$\mathcal L(X)$ as $\tau\to0$:
\begin{equation}\label{eq:ode:p_local_defect}
S^\sigma(\tau)-T(\tau)
=
\frac{\tau^2}{2}
\sum_{1\le i<j\le p}
\Xi^{ij}(\sigma)[A_i,A_j]
+\mathcal R_\sigma(\tau),
\end{equation}
where $[A_i,A_j]:=A_iA_j-A_jA_i$ and, for all sufficiently small $\tau$,
\begin{equation}\label{eq:ode:p_local_defect_remainder}
\|\mathcal R_\sigma(\tau)\|_{\mathcal L(X)}
\le C\tau^3.
\end{equation}
The constant $C$ depends on $p$ and
$\max_m\|A_m\|_{\mathcal L(X)}$, but is independent of $\sigma$ and $\tau$.
\end{lemma}
% \begin{lemma}[Pairwise local defect for $p$ bounded operators]\label{lem:ode:p_local_defect}
% For every $\sigma\in\mathfrak S_p$, one has, as $\tau\to0$ in $\mathcal L(X)$,
% \begin{equation}\label{eq:ode:p_local_defect}
% S^\sigma(\tau)-T(\tau)
% =
% \frac{\tau^2}{2}\sum_{1\le i<j\le p}\Xi^{ij}(\sigma)[A_i,A_j]
% +\mathcal R_\sigma(\tau),
% \end{equation}
% where $[A_i,A_j]:=A_iA_j-A_jA_i$ and
% \begin{equation}\label{eq:ode:p_local_defect_remainder}
% \|\mathcal R_\sigma(\tau)\|_{\mathcal L(X)}\le C\tau^3.
% \end{equation}
% The constant $C$ depends on $p$ and on $\max_m\|A_m\|_{\mathcal L(X)}$, but is independent of $\sigma$ and $\tau$.
% \end{lemma}

\begin{proof}
Since the operators are bounded,
\[
e^{\tau A_m}=I+\tau A_m+\frac{\tau^2}{2}A_m^2+O_{\mathcal L(X)}(\tau^3).
\]
Multiplying the $p$ factors gives
\begin{equation}\label{eq:ode:p_split_expansion}
S^\sigma(\tau)
=
I+\tau\sum_{m=1}^{p}A_m
+\frac{\tau^2}{2}\sum_{m=1}^{p}A_m^2
+\tau^2\sum_{1\le\alpha<\beta\le p}A_{\sigma_\beta}A_{\sigma_\alpha}
+O_{\mathcal L(X)}(\tau^3).
\end{equation}
The remainder is uniform in $\sigma$ because $\mathfrak S_p$ is finite and all $A_m$ are bounded. On the other hand,
\begin{align}
T(\tau)
&=e^{\tau\sum_{m=1}^{p}A_m}\notag\\
&=I+\tau\sum_{m=1}^{p}A_m
+\frac{\tau^2}{2}\sum_{m=1}^{p}A_m^2
+\frac{\tau^2}{2}\sum_{1\le i<j\le p}(A_iA_j+A_jA_i)
+O_{\mathcal L(X)}(\tau^3).
\label{eq:ode:p_exact_expansion}
\end{align}
For a fixed pair $i<j$, let $B_{ij}^\sigma$ denote its contribution to the cross term in \eqref{eq:ode:p_split_expansion}. Then
\begin{equation}\label{eq:ode:Bij_definition}
B_{ij}^{\sigma}:=
\begin{cases}
A_iA_j, & \operatorname{pos}_{\sigma}(j)<\operatorname{pos}_{\sigma}(i),\\[1mm]
A_jA_i, & \operatorname{pos}_{\sigma}(i)<\operatorname{pos}_{\sigma}(j).
\end{cases}
\end{equation}
By Definition~\ref{def:pairwise_sign},
\begin{equation}\label{eq:ode:Bij_pairwise_identity}
B_{ij}^{\sigma}
=
\frac12(A_iA_j+A_jA_i)
+\frac12\Xi^{ij}(\sigma)[A_i,A_j].
\end{equation}
Substitution into \eqref{eq:ode:p_split_expansion} and comparison with \eqref{eq:ode:p_exact_expansion} yield \eqref{eq:ode:p_local_defect}.
\end{proof}

The preceding lemma shows that the leading $O(\tau^2)$ term in the local
splitting defect is determined by the pairwise signs
$\Xi^{ij}(\sigma)$, which encode the relative ordering of the operators.
We now combine this expansion with the logarithmic cancellation estimate
from Section~\ref{sec:prelim:pairwise_cancellation}.

\begin{theorem}[Convergence for $p$ bounded linear operators]\label{thm:ode:p_convergence}
Let $\sigma^n$ be generated by the quasi-random ordering rule \eqref{eq:setup:interval_permutation_rule}, and let
\begin{equation}\label{eq:ode:p_qr_scheme}
u_{n+1}
=
S^{\sigma^n}(\tau)u_n,
\qquad
n=0,\dots,N-1,
\qquad
\tau=\frac{T}{N}.
\end{equation}
Then, for every fixed $p$, there exist $N_0\ge2$ and $C>0$, independent of $N$, such that, for all $N\ge N_0$,
\begin{equation}\label{eq:ode:p_error_tau}
\max_{0\le n\le N}\|u_n-u(t_n)\|_X
\le
C\tau^2(1+\log N).
\end{equation}
Equivalently, since $\tau=T/N$, after adjusting the constant,
\begin{equation}\label{eq:ode:p_error_N}
\max_{0\le n\le N}\|u_n-u(t_n)\|_X
\le
C\frac{1+\log N}{N^2}.
\end{equation}
\end{theorem}

\begin{proof}
Set
\[
e_n:=u_n-u(t_n),
\qquad
E_n:=S^{\sigma^n}(\tau).
\]
Because each time step contains every bounded operator exactly once,
\begin{equation}\label{eq:ode:p_En_bound}
\|E_n\|_{\mathcal L(X)}
\le
e^{\tau\sum_{m=1}^p\|A_m\|}
\le
1+C\tau
\end{equation}
for sufficiently small $\tau$, uniformly in $n$. Moreover,
\begin{equation}\label{eq:ode:p_En_identity}
\|E_n-I\|_{\mathcal L(X)}\le C\tau.
\end{equation}

By Lemma~\ref{lem:ode:p_local_defect}, applied to $u(t_n)$,
\begin{equation}\label{eq:ode:p_error_recursion}
e_{n+1}
=
E_ne_n
+
\frac{\tau^2}{2}
\sum_{1\le i<j\le p}
\xi_n^{ij}[A_i,A_j]u(t_n)
+
d_n,
\end{equation}
where
\[
\xi_n^{ij}
=
\Xi^{ij}(\sigma^n),
\qquad
\|d_n\|_X\le C\tau^3.
\]

For $0\le k\le n\le N-1$, define
\begin{equation}\label{eq:ode:p_Gnk}
G_{n,k}
:=
\begin{cases}
E_nE_{n-1}\cdots E_{k+1}, & k\le n-1,\\
I, & k=n,
\end{cases}
\end{equation}
and
\begin{equation}\label{eq:ode:p_gnk_definition}
g_{n,k}^{ij}
:=
G_{n,k}[A_i,A_j]u(t_k).
\end{equation}
Iterating \eqref{eq:ode:p_error_recursion} gives
\begin{equation}\label{eq:ode:p_error_representation}
e_{n+1}
=
\frac{\tau^2}{2}
\sum_{1\le i<j\le p}
\sum_{k=0}^{n}
g_{n,k}^{ij}\xi_k^{ij}
+
\sum_{k=0}^{n}G_{n,k}d_k.
\end{equation}

From \eqref{eq:ode:p_En_bound},
\begin{equation}\label{eq:ode:p_G_bound}
\|G_{n,k}\|_{\mathcal L(X)}
\le
(1+C\tau)^{n-k}
\le
e^{CT}.
\end{equation}
Since $u'(t)=A_\Sigma u(t)$ and the operators are bounded, for every fixed pair $i<j$ there are constants $M_{ij},L_{ij}>0$, independent of $n,k,N$, such that
\begin{equation}\label{eq:ode:p_g_bound}
\|g_{n,k}^{ij}\|_X\le M_{ij}
\end{equation}
and
\begin{equation}\label{eq:ode:p_g_lip}
\|g_{n,k+1}^{ij}-g_{n,k}^{ij}\|_X
\le
L_{ij}\tau.
\end{equation}
Indeed,
\[
g_{n,k+1}^{ij}-g_{n,k}^{ij}
=
G_{n,k+1}
\left(
[A_i,A_j](u(t_{k+1})-u(t_k))
+
(I-E_{k+1})[A_i,A_j]u(t_k)
\right),
\]
and both terms in parentheses are $O(\tau)$ uniformly on $[0,T]$.

For a fixed pair $i<j$, set
\[
P_m^{ij}
:=
\sum_{k=0}^{m-1}\xi_k^{ij},
\qquad
P_0^{ij}:=0.
\]
Theorem~\ref{thm:prelim:pairwise_cancellation} yields
\begin{equation}\label{eq:ode:p_partial_sum_qr}
\max_{1\le m\le N}|P_m^{ij}|
\le
C_{p,R}\log N.
\end{equation}
Applying Abel summation,
\begin{equation}\label{eq:ode:p_abel}
\sum_{k=0}^{n}g_{n,k}^{ij}\xi_k^{ij}
=
g_{n,n}^{ij}P_{n+1}^{ij}
-
\sum_{k=0}^{n-1}
\bigl(g_{n,k+1}^{ij}-g_{n,k}^{ij}\bigr)P_{k+1}^{ij}.
\end{equation}
Therefore, by \eqref{eq:ode:p_g_bound}, \eqref{eq:ode:p_g_lip}, and \eqref{eq:ode:p_partial_sum_qr},
\begin{align}
\left\|
\sum_{k=0}^{n}g_{n,k}^{ij}\xi_k^{ij}
\right\|_X
&\le
C\log N
\left(
M_{ij}+\sum_{k=0}^{n-1}L_{ij}\tau
\right)
\le
C\log N.
\label{eq:ode:p_osc_bound}
\end{align}
Since the number of pairs is fixed, the total principal contribution in \eqref{eq:ode:p_error_representation} is bounded by $C\tau^2\log N$.

Finally, \eqref{eq:ode:p_G_bound} and $\|d_k\|_X\le C\tau^3$ imply
\[
\left\|
\sum_{k=0}^{n}G_{n,k}d_k
\right\|_X
\le
C\sum_{k=0}^{n}\tau^3
\le
C\tau^2.
\]
Substituting these two estimates into \eqref{eq:ode:p_error_representation} proves \eqref{eq:ode:p_error_tau}; \eqref{eq:ode:p_error_N} follows from $\tau=T/N$.
\end{proof}

\begin{remark}[Dependence on the number of operators]
\label{rem:ode:p_cost_constant}
The estimate above is a fixed-$p$ result. Under the equivalent lexicographic
interval representation used in the analysis, the bound
\[
V(h_{ij})\le 2(p!-1)
\]
leads to a convergence constant that is not uniform in $p$. The theorem
therefore describes the convergence rate as $\tau\to0$ for each fixed number
of split operators. The lexicographic ordering is used only to define the
equivalent interval representation; other enumerations preserve the
mean-zero pairwise identity but may lead to different variation bounds.
\end{remark}

\subsection{Abstract convergence analysis for three-operator splitting}\label{sec:general_convergence}

% We next formulate an abstract convergence framework for three-operator
% splitting. The framework is stated at the level of the associated flows and
% applies whenever the local defect, stability, and regularity assumptions
% below are satisfied. In particular, we assume a pairwise expansion of the
% leading $O(\tau^2)$ local defect, stable one-step propagation of the error,
% and sufficient temporal regularity of the propagated defect coefficients.
% Under these assumptions, the pairwise logarithmic cancellation established
% in Theorem~\ref{thm:prelim:pairwise_cancellation} yields an essentially
% second-order global error. A nonlinear PDE application of this framework is
% given in the next subsection.
We now establish an abstract convergence result for three-operator
splitting schemes. The analysis is formulated in terms of the
associated flows and relies on three key properties: a pairwise
expansion of the leading $O(\tau^2)$ local splitting defect,
stability of the one-step error propagation, and sufficient temporal
regularity of the propagated defect coefficients. Combining these
properties with the pairwise logarithmic cancellation established in
Theorem~\ref{thm:prelim:pairwise_cancellation}, we obtain an
essentially second-order global error estimate. In the next
subsection, we verify these assumptions for the three-operator
Allen--Cahn equation with a background flow.

Let $X$ be a Banach space and let $Y\hookrightarrow X$ be a more regular space.  Consider three evolution operators $\mathcal F_1,\mathcal F_2,\mathcal F_3$ with corresponding subflows $S_1,S_2,S_3$, and let $T(t)$ denote the exact flow generated by $\mathcal F_1+\mathcal F_2+\mathcal F_3$ on the class of states under consideration.  For $p=3$, let
\[
\sigma^n\in\mathfrak S_3,
\qquad n=0,\dots,N-1,
\]
be generated by the quasi-random ordering rule \eqref{eq:setup:interval_permutation_rule}, and define
\[
\xi_n^{ij}:=\Xi^{ij}(\sigma^n),
\qquad 1\le i<j\le3.
\]  The numerical scheme is
\begin{equation}\label{eq:general:scheme}
 u_{n+1}=S^{\sigma^n}(\tau)[u_n],
 \qquad n=0,\dots,N-1,
\end{equation}
and we write
\begin{equation}\label{eq:general:error}
 e_n:=u_n-u(t_n).
\end{equation}

\begin{lemma}[Discrete Gr\"onwall inequality]\label{lem:general:discrete_gronwall}
Let $\{a_n\}_{n\ge0}$ be a nonnegative sequence satisfying
\[
 a_n\le\alpha+\beta\tau\sum_{k=0}^{n-1}a_k.
\]
Then $a_n\le\alpha e^{\beta n\tau}$.
\end{lemma}

\begin{theorem}[Abstract convergence of three-operator quasi-random splitting]\label{thm:general:abstract_convergence}
Let $B\subset Y$ be a bounded set containing the exact trajectory and all
numerical and intermediate states arising in the estimates below. Assume
that the following properties hold uniformly on $B$ for all sufficiently
small $\tau$.

\medskip
\noindent\textbf{(i) Pairwise local defect.}
There exist mappings $\Phi_{ij}:Y\to X$, $1\le i<j\le3$, such that for every $\sigma\in\mathfrak S_3$ and $a\in B$,
\begin{equation}\label{eq:general:local_pairwise}
 S^\sigma(\tau)[a]-T(\tau)[a]
 =
 \frac{\tau^2}{2}
 \sum_{1\le i<j\le3}
 \Xi^{ij}(\sigma)\Phi_{ij}(a)
 +r_\sigma(\tau,a),
\end{equation}
with
\begin{equation}\label{eq:general:local_remainder}
 \|r_\sigma(\tau,a)\|_X\le C_B\tau^3.
\end{equation}

\medskip
\noindent\textbf{(ii) One-step propagation.}
For each $n$ there exists a bounded linear operator $E_n:X\to X$ such that
\begin{equation}\label{eq:general:stability}
 S^{\sigma^n}(\tau)[u_n]-S^{\sigma^n}(\tau)[u(t_n)]
 =E_ne_n+\rho_n,
\end{equation}
where
\begin{equation}\label{eq:general:En_bounds}
 \|E_n\|_{\mathcal L(X)}\le1+C_B\tau,
\end{equation}
and
\begin{equation}\label{eq:general:stability_remainder}
 \|\rho_n\|_X\le C_B\tau\|e_n\|_X.
\end{equation}
Here $E_n$ represents the principal linear part of the one-step propagation
of the error; it is not the numerical splitting operator itself.
% The operator $E_n$ is only an analytical principal part of the propagation of an
% already existing error.  It is not the numerical splitting step itself: the
% actual update remains the complete composition $S^{\sigma^n}(\tau)$ containing
% all three subflows.  How $E_n$ and $\rho_n$ are constructed is problem dependent;
% the theorem only requires the bounds above.

\medskip
\noindent\textbf{(iii) Regularity of the propagated principal coefficients.}
For $0\le k\le n\le N-1$, define
\begin{equation}\label{eq:general:propagator}
G_{n,k}:=
\begin{cases}
 E_nE_{n-1}\cdots E_{k+1},&k\le n-1,\\
 I,&k=n,
\end{cases}
\end{equation}
and, for $1\le i<j\le3$,
\begin{equation}\label{eq:general:g_n}
 g_n^{ij}(t_k):=G_{n,k}\Phi_{ij}(u(t_k)).
\end{equation}
Assume that there exist constants $M,L>0$, independent of $n,k,N$, such that
\begin{equation}\label{eq:general:g_bound}
 \|g_n^{ij}(t_k)\|_X\le M
\end{equation}
and
\begin{equation}\label{eq:general:g_lip}
 \|g_n^{ij}(t_{k+1})-g_n^{ij}(t_k)\|_X\le L\tau
\end{equation}
for every pair $1\le i<j\le3$.

Then there exist $N_0\ge2$ and $C>0$, independent of $N$, such that, for
all $N\ge N_0$,
\begin{equation}\label{eq:general:abstract_error}
 \max_{0\le n\le N}\|e_n\|_X
 \le C\tau^2(1+\log N).
\end{equation}
Equivalently, since $\tau=T/N$, after adjusting the constant,
\begin{equation}\label{eq:general:abstract_error_N}
 \max_{0\le n\le N}\|e_n\|_X
 \le C\frac{1+\log N}{N^2}.
\end{equation}
\end{theorem}

\begin{proof}
By \eqref{eq:general:local_pairwise} and \eqref{eq:general:stability},
\begin{equation}\label{eq:general:error_recursion}
 e_{n+1}
 =E_ne_n
 +\frac{\tau^2}{2}
 \sum_{1\le i<j\le3}
 \xi_n^{ij}\Phi_{ij}(u(t_n))
 +R_n,
\end{equation}
where
\[
 R_n=r_{\sigma^n}(\tau,u(t_n))+\rho_n,
 \qquad
 \|R_n\|_X\le C(\tau^3+\tau\|e_n\|_X).
\]
Iterating \eqref{eq:general:error_recursion} gives
\begin{equation}\label{eq:general:error_representation}
 e_{n+1}
 =
 \frac{\tau^2}{2}
 \sum_{1\le i<j\le3}
 \sum_{k=0}^{n}
 g_n^{ij}(t_k)\xi_k^{ij}
 +\sum_{k=0}^{n}G_{n,k}R_k.
\end{equation}

Fix $1\le i<j\le3$ and define
\[
 H_m^{ij}:=\sum_{\ell=0}^{m-1}\xi_\ell^{ij},
 \qquad H_0^{ij}:=0.
\]
Theorem~\ref{thm:prelim:pairwise_cancellation}, specialized to $p=3$, gives
\begin{equation}\label{eq:general:pairwise_partial_sum}
 \max_{0\le m\le N}|H_m^{ij}|
 \le C_{3,R}\log N.
\end{equation}
Abel summation yields
\begin{align}
\sum_{k=0}^{n}g_n^{ij}(t_k)\xi_k^{ij}
&=
 g_n^{ij}(t_n)H_{n+1}^{ij}
 -\sum_{k=0}^{n-1}
 \bigl(g_n^{ij}(t_{k+1})-g_n^{ij}(t_k)\bigr)H_{k+1}^{ij}.
\label{eq:general:pairwise_abel}
\end{align}
Using \eqref{eq:general:g_bound}, \eqref{eq:general:g_lip}, and \eqref{eq:general:pairwise_partial_sum}, we obtain
\begin{align*}
\left\|
 \sum_{k=0}^{n}g_n^{ij}(t_k)\xi_k^{ij}
\right\|_X
&\le
 C\log N
 \left(
 M+\sum_{k=0}^{n-1}L\tau
 \right)\\
&\le C(M+LT)\log N.
\end{align*}
Since there are only three pairs,
\begin{equation}\label{eq:general:main_term_bound}
\left\|
 \frac{\tau^2}{2}
 \sum_{1\le i<j\le3}
 \sum_{k=0}^{n}g_n^{ij}(t_k)\xi_k^{ij}
\right\|_X
\le C\tau^2\log N.
\end{equation}

Property (ii) gives
\[
 \|G_{n,k}\|_{\mathcal L(X)}
 \le(1+C\tau)^{n-k}
 \le e^{CT}.
\]
Consequently,
\[
\left\|\sum_{k=0}^{n}G_{n,k}R_k\right\|_X
\le
 C\sum_{k=0}^{n}\bigl(\tau^3+\tau\|e_k\|_X\bigr)
\le
 C\tau^2+C\tau\sum_{k=0}^{n}\|e_k\|_X.
\]
Combining this estimate with \eqref{eq:general:main_term_bound} gives
\[
 \|e_{n+1}\|_X
 \le
 C\tau^2(1+\log N)
 +C\tau\sum_{k=0}^{n}\|e_k\|_X.
\]
Lemma~\ref{lem:general:discrete_gronwall} completes the proof.
\end{proof}

\subsection{Application to the three-operator Allen--Cahn equation with a background flow}
\label{sec:pde}

We now verify the assumptions of Theorem~\ref{thm:general:abstract_convergence} for the Allen--Cahn equation with an incompressible background flow,
\begin{equation}\label{eq:ac3-pde}
 \partial_t u+v(x)\cdot\nabla u=\nu\Delta u-f(u),
 \qquad f(u)=u^3-u,\qquad u|_{t=0}=u_0,
\end{equation}
posed on $\Omega=\mathbb T^d=[0,L]^d$, $d\geq2$, with periodic boundary conditions and $\nu>0$. Throughout this subsection, the velocity field is independent of time and satisfies
\begin{equation}\label{eq:ac3-velocity}
 v\in C^\infty(\mathbb T^d;\mathbb R^d),\qquad \nabla\cdot v=0.
\end{equation}
Following the three-way decomposition in \cite{li2025convergence}, we take
\begin{equation}\label{eq:ac3-operators}
 \mathcal A_1=\cA,\quad \mathcal A_2=\cL,\quad \mathcal A_3=\cR,\qquad
 \cA u=-v\cdot\nabla u,\quad \cL u=\nu\Delta u,\quad \cR(u)=u-u^3.
\end{equation}
The corresponding subflows are $S_1=S_{\cA}$, $S_2=S_{\cL}$, and $S_3=S_{\cR}$, and the exact flow is denoted by $T(t)$. Let $u_n$ be generated by \eqref{eq:general:scheme}, with the orderings generated by the quasi-random ordering rule \eqref{eq:setup:interval_permutation_rule}. For $k\in\mathbb N$ and $r\in[2,\infty]$, we work with
\begin{equation}\label{eq:ac3-spaces}
 X=W^{k,r}(\Omega),\qquad Y=W^{k+6,r}(\Omega)\cap W^{k+5,\infty}(\Omega).
\end{equation}
All calculations are first carried out for smooth data. The estimates depend only on the stated Sobolev norms, and the passage to general initial data is given at the end of the subsection.

\begin{lemma}[Sobolev estimates for the three subflows]\label{lem:ac3-subflows}
Let $m\in\mathbb N$ and $r\in[2,\infty]$. Under \eqref{eq:ac3-velocity},
\begin{equation}\label{eq:ac3-linear}
 \|S_{\cA}(t)w\|_{m,r}\leq e^{C_m t}\|w\|_{m,r},\qquad
 \|S_{\cL}(t)w\|_{m,r}\leq\|w\|_{m,r}.
\end{equation}
For $m\geq1$, if $\|w\|_{m-1,r}+\|w\|_{m-1,\infty}\leq M$, then there exist $t_0>0$ and $C_M>0$ such that, for $0\leq t\leq t_0$,
\begin{equation}\label{eq:ac3-reaction}
 \|S_{\cR}(t)w\|_{m,r}\leq e^{C_Mt}\|w\|_{m,r}+C_Mt.
\end{equation}
At order zero,
\begin{equation}\label{eq:ac3-zero}
 \|S_{\cR}(t)w\|_{0,r}\leq e^t\|w\|_{0,r},\qquad
 \|S_{\cR}(t)w\|_{0,\infty}\leq\max\{1,\|w\|_{0,\infty}\}.
\end{equation}
\end{lemma}
\begin{proof}
For the advection flow $q_t=-v\cdot\nabla q$, differentiation gives
\[
 \partial_tD^\alpha q+v\cdot\nabla D^\alpha q
 =-[D^\alpha,v\cdot\nabla]q.
\]
The transport term vanishes in the $L^r$ energy estimate because $\nabla\cdot v=0$, while the commutator contains derivatives of $q$ of order at most $|\alpha|$. For $r<\infty$, summing the $r$-th-power energy estimates over $|\alpha|\leq m$ and applying H\"older's inequality yields
\[
 \frac{d}{dt}\|q(t)\|_{m,r}\leq C_m\|q(t)\|_{m,r}.
\]
For $r=\infty$, the characteristic estimates for the finite family of derivatives give the same inequality for the upper right Dini derivative of $\max_{|\alpha|\leq m}\|D^\alpha q\|_\infty$. Thus the estimates use precisely the norms \eqref{eq:setup:Sobolev}--\eqref{eq:setup:Sobolev_infty}. Gronwall's inequality proves the first bound in \eqref{eq:ac3-linear}. The heat estimate follows from commutation with spatial derivatives and the $L^r$ contractivity of the heat semigroup.

For the reaction flow, $q_t=q-q^3$ gives \eqref{eq:ac3-zero}. For $|\alpha|\geq1$,
\begin{equation}\label{eq:ac3-reactionder}
 \partial_tD^\alpha q=(1-3q^2)D^\alpha q+F_\alpha(q),
\end{equation}
where every derivative occurring in $F_\alpha$ has order strictly less than $|\alpha|$. Induction first controls the lower-order norms on a time interval depending only on $M$. Taking one factor in $L^r$ and the remaining factors in $L^\infty$, we obtain
\[
 \frac{d}{dt}\|q(t)\|_{m,r}\leq C_M\|q(t)\|_{m,r}+C_M.
\]
This proves \eqref{eq:ac3-reaction}. The argument also applies to $r=\infty$ through the scalar variation-of-constants formula in \eqref{eq:ac3-reactionder}.
\end{proof}

\begin{proposition}[Uniform Sobolev bounds]\label{prop:ac3-bounds}
Assume $u_0\in Y$. There exist $\tau_0>0$ and $M>0$, independent of $N$ and the ordering sequence, such that, for $\tau\leq\tau_0$,
\begin{equation}\label{eq:ac3-uniform}
 \sup_{0\leq t\leq T}\|u(t)\|_Y
 +\max_{0\leq n\leq N}\|u_n\|_Y\leq M.
\end{equation}
The same type of bound holds for every intermediate subflow state.
\end{proposition}
\begin{proof}
The maximum principle for \eqref{eq:ac3-pde} gives
\[
 \sup_{0\leq t\leq T}\|u(t)\|_{0,\infty}
 \leq\max\{1,\|u_0\|_{0,\infty}\}.
\]
For the splitting solution and its intermediate states, the same bound follows from the transport and heat maximum principles and \eqref{eq:ac3-zero}.

Differentiating \eqref{eq:ac3-pde}, we obtain, for $|\alpha|\geq1$,
\begin{align*}
 \partial_tD^\alpha u+v\cdot\nabla D^\alpha u
 &=\nu\Delta D^\alpha u+(1-3u^2)D^\alpha u\\
 &\quad+F_\alpha(u)-[D^\alpha,v\cdot\nabla]u.
\end{align*}
The diffusion contribution is nonpositive, and the advection commutator is bounded by $C\|u\|_{m,r}$ at derivative order $m$. The lower-order reaction terms therefore give
\begin{equation}\label{eq:ac3-induction}
 \frac{d}{dt}\|u(t)\|_{m,r}
 \leq C\|u(t)\|_{m,r}
 +C\bigl(1+\|u(t)\|_{m-1,\infty}^2\bigr)\|u(t)\|_{m-1,r}.
\end{equation}
Here the componentwise estimates are combined in $\ell^r$. For $r=\infty$, the maximum principle for the differentiated equations gives the corresponding inequality for the upper right Dini derivative of the maximum norm. Starting from the zero-order bound, we first obtain the $W^{m,\infty}$ estimates through $m=k+5$. These also bound the lower-order $L^r$ norms on the torus. Applying \eqref{eq:ac3-induction} at $m=k+6$ and using Gronwall's inequality gives the exact-solution bound.

For the numerical solution, suppose the lower-order norms of all grid and intermediate states are bounded. Lemma~\ref{lem:ac3-subflows} implies, for every permutation,
\begin{equation}\label{eq:ac3-discrete}
 \|u_{n+1}\|_{m,r}\leq e^{C\tau}\|u_n\|_{m,r}+C\tau.
\end{equation}
The same subflow estimates control the intermediate states at order $m$. Discrete Gronwall and induction first give the $W^{m,\infty}$ bounds up to $k+5$, and then the $W^{k+6,r}$ bound. The constants depend only on the fixed terminal time, the coefficients, and the initial norms.

This argument also gives uniform bounds for short exact trajectories and partial-time subflow compositions starting from any bounded subset of $Y$. We enlarge the bounded set under consideration to include these states, in particular subflow states starting at the exact values $u(t_n)$.
\end{proof}

\begin{lemma}[Pairwise expansion of the local defect]\label{lem:ac3-local}
Let $B\subset Y$ be bounded. For every $a\in B$ and $\sigma\in\mathfrak S_3$,
\begin{equation}\label{eq:ac3-local}
 S^\sigma(\tau)[a]-T(\tau)[a]
 =\frac{\tau^2}{2}\sum_{1\leq i<j\leq3}\Xi^{ij}(\sigma)\Phi_{ij}(a)
 +r_\sigma(\tau,a),
\end{equation}
where
\begin{equation}\label{eq:ac3-bracket}
 \Phi_{ij}(a)=D\mathcal A_i(a)\mathcal A_j(a)-D\mathcal A_j(a)\mathcal A_i(a),\qquad
 \sup_{\sigma\in\mathfrak S_3}\|r_\sigma(\tau,a)\|_{k,r}\leq C_B\tau^3.
\end{equation}
For the operators in \eqref{eq:ac3-operators},
\begin{align}
 \Phi_{12}(a)&=\cA\cL a-\cL\cA a
 =\nu(\Delta v)\cdot\nabla a
 +2\nu\sum_{q=1}^d(\partial_qv)\cdot\nabla(\partial_q a),\label{eq:ac3-phi12}\\
 \Phi_{13}(a)&=\cA\cR(a)-D\cR(a)\cA a=0,\label{eq:ac3-phi13}\\
 \Phi_{23}(a)&=\cL\cR(a)-D\cR(a)\cL a=-6\nu a|\nabla a|^2.\label{eq:ac3-phi23}
\end{align}
\end{lemma}
\begin{proof}
Set $F_1=\cA$, $F_2=\cL$, and $F_3=\cR$. We first justify the third-order remainder for the complete composition. For $j=1,2,3$, write
\[
 H_j(s_1,\ldots,s_j)
 =S_{\sigma_j}(s_j)\cdots S_{\sigma_1}(s_1)a,\qquad H_0=a,
\]
and introduce
\begin{equation}\label{eq:ac3-scale}
 Z_q=W^{k+6-2q,r}\cap W^{k+5-2q,\infty},\quad q=0,1,2,
 \qquad Z_3=W^{k,r}.
\end{equation}
For $0\leq s_i\leq\tau_0$, we claim that every mixed time derivative $\partial^\beta H_j$ of total order $q=|\beta|\leq3$ is uniformly bounded in $Z_q$. The order-zero assertion follows from Proposition~\ref{prop:ac3-bounds}, applied uniformly to $a\in B$ and its short subflow trajectories.

We proceed by induction over the subflows. For a heat-flow stage,
\[
 \partial_s^\ell\partial^\gamma\bigl(S_{\cL}(s)H_{j-1}\bigr)
 =\cL^\ell S_{\cL}(s)\partial^\gamma H_{j-1}.
\]
The loss of $2\ell$ spatial derivatives is exactly allowed by \eqref{eq:ac3-scale}. For a transport stage the corresponding formula contains $\cA^\ell S_{\cA}(s)$, which loses only $\ell$ derivatives. The estimates in \eqref{eq:ac3-linear} prove the claim in both cases.

For reaction, write $S_{\cR}(s)w=\mathcal F(s,w)$, where
\[
 \mathcal F(s,z)=\frac{z}{\sqrt{e^{-2s}+(1-e^{-2s})z^2}}.
\]
The scalar function $\mathcal F$ is smooth with uniformly bounded derivatives on the relevant bounded sets. At positive total time-derivative order $q$, the chain rule gives terms of the form
\begin{equation}\label{eq:ac3-chain}
 c(s,H_{j-1})\prod_{a=1}^M\partial^{\beta_a}H_{j-1},
 \qquad |\beta_a|\geq1,\qquad \sum_{a=1}^M|\beta_a|\leq q,
\end{equation}
together with pure explicit-time derivatives of $\mathcal F(s,H_{j-1})$.
For $q=1,2$, put $m=k+6-2q$. If an order-$q$ factor occurs in \eqref{eq:ac3-chain}, it is the only differentiated factor; it is bounded in $W^{m,r}$ and its coefficient in $W^{m,\infty}$. Otherwise every factor has order at most $q-1$ and has at least $k+5-2(q-1)=m+1$ bounded spatial derivatives. Leibniz' rule therefore gives the required $W^{m,r}$ and $W^{m-1,\infty}$ bounds. For $q=3$, an order-three factor again occurs alone and is multiplied by a $W^{k,\infty}$ coefficient. All other factors have time-derivative order at most two and belong to $W^{k+2,r}\cap W^{k+1,\infty}$, which suffices for their product in $W^{k,r}$. Pure explicit-time terms are controlled by scalar composition and the order-zero bounds. At order zero, a spatial derivative of order $k+6$ appears linearly; all other differentiated factors have order at most $k+5$. This completes the induction.

Consequently, $h_\sigma(t)=H_3(t,t,t)$ satisfies
\[
 \sup_{\sigma\in\mathfrak S_3}\sup_{0\leq t\leq\tau_0}
 \|h_\sigma'''(t)\|_{k,r}\leq C_B.
\]
Taylor's formula with integral remainder now gives
\begin{align}\label{eq:ac3-splitexp}
 S^\sigma(\tau)[a]
 &=a+\tau\sum_{i=1}^3F_i(a)
 +\frac{\tau^2}{2}\sum_{i=1}^3DF_i(a)F_i(a)\notag\\
 &\quad+\tau^2\sum_{1\leq\alpha<\beta\leq3}
 DF_{\sigma_\beta}(a)F_{\sigma_\alpha}(a)
 +R_S^\sigma(\tau,a),
\end{align}
with $\sup_\sigma\|R_S^\sigma(\tau,a)\|_{k,r}\leq C_B\tau^3$.

For the exact flow, let $F=F_1+F_2+F_3$ and $q(t)=T(t)[a]$. The uniform $Y$ bound gives
\[
 q_t=F(q)\in Z_1,\qquad
 q_{tt}=(\cA+\cL)q_t+\cR'(q)q_t\in Z_2,
\]
uniformly on the short trajectory. Moreover,
\[
 q_{ttt}=(\cA+\cL)q_{tt}+\cR'(q)q_{tt}+\cR''(q)(q_t)^2,
 \qquad \cR''(q)=-6q.
\]
The first term is bounded in $W^{k,r}$ by $q_{tt}\in Z_2$, and the products are bounded by $q_t\in Z_1$ and Leibniz' rule. Hence
\begin{equation}\label{eq:ac3-exactexp}
 T(\tau)[a]=a+\tau F(a)+\frac{\tau^2}{2}DF(a)F(a)+R_T(\tau,a),
 \qquad \|R_T(\tau,a)\|_{k,r}\leq C_B\tau^3.
\end{equation}
For each $i<j$, the cross term in \eqref{eq:ac3-splitexp} is $DF_i(a)F_j(a)$ if $j$ precedes $i$, and $DF_j(a)F_i(a)$ otherwise. Subtracting \eqref{eq:ac3-exactexp} and using Definition~\ref{def:pairwise_sign} yields \eqref{eq:ac3-local}--\eqref{eq:ac3-bracket}.

It remains to compute the coefficients. Expanding $\Delta(v\cdot\nabla a)$ gives \eqref{eq:ac3-phi12}. Since $\cA$ is a derivation,
\[
 \cA\cR(a)=\cA(a-a^3)=(1-3a^2)\cA a=D\cR(a)\cA a,
\]
which proves \eqref{eq:ac3-phi13}. Finally,
\[
 \cL\cR(a)-D\cR(a)\cL a=\nu\Delta(a-a^3)-(1-3a^2)\nu\Delta a
 =-6\nu a|\nabla a|^2,
\]
giving \eqref{eq:ac3-phi23}.
\end{proof}

\begin{lemma}[Principal propagation decomposition]\label{lem:ac3-propagation}
For $\sigma\in\mathfrak S_3$, write
\begin{equation}\label{eq:ac3-factor}
 S^\sigma(\tau)=P_\sigma^+S_{\cR}(\tau)P_\sigma^-,\qquad
 E^\sigma(\tau)=P_\sigma^+P_\sigma^-,
\end{equation}
where $P_\sigma^-$ and $P_\sigma^+$ contain the advection and diffusion subflows before and after reaction, respectively. Either composition may be the identity. Then
\begin{equation}\label{eq:ac3-E}
 E^\sigma(\tau)\in\{S_{\cL}(\tau)S_{\cA}(\tau),\ S_{\cA}(\tau)S_{\cL}(\tau)\},\qquad
 \|E^\sigma(\tau)\|_{\mathcal L(X)}\leq e^{C\tau}\leq1+C'\tau.
\end{equation}
On the bounded set supplied by Proposition~\ref{prop:ac3-bounds},
\begin{equation}\label{eq:ac3-propagation}
 S^\sigma(\tau)[a]-S^\sigma(\tau)[b]
 =E^\sigma(\tau)(a-b)+\rho_\sigma(a,b),\qquad
 \|\rho_\sigma(a,b)\|_{k,r}\leq C\tau\|a-b\|_{k,r}.
\end{equation}
\end{lemma}
\begin{proof}
The linear estimate follows from \eqref{eq:ac3-linear}. Set $x=P_\sigma^-a$, $y=P_\sigma^-b$, $q_x(s)=S_{\cR}(s)x$, and $q_y(s)=S_{\cR}(s)y$. Since
\[
 \cR(w_1)-\cR(w_2)=(1-w_1^2-w_1w_2-w_2^2)(w_1-w_2),
\]
the uniform $W^{k,\infty}$ bounds and the multiplication estimate in $W^{k,r}$ give
\[
 \|\cR(w_1)-\cR(w_2)\|_{k,r}\leq C\|w_1-w_2\|_{k,r}.
\]
Gronwall's inequality implies $\|q_x(s)-q_y(s)\|_{k,r}\leq e^{Cs}\|x-y\|_{k,r}$. Thus
\begin{align*}
 S_{\cR}(\tau)x-S_{\cR}(\tau)y
 &=(x-y)+\delta_R(x,y),\\
 \delta_R(x,y)&=\int_0^\tau\bigl(\cR(q_x(s))-\cR(q_y(s))\bigr)\,ds,
 \qquad \|\delta_R(x,y)\|_{k,r}\leq C\tau\|x-y\|_{k,r}.
\end{align*}
Applying $P_\sigma^+$ yields \eqref{eq:ac3-propagation} with
$\rho_\sigma=P_\sigma^+\delta_R$.
Here $E^\sigma$ preserves the order in which the advection and diffusion
subflows appear in the original composition, while the reaction correction
is included in $\rho_\sigma$.
\end{proof}

\begin{lemma}[Near-identity estimate for the principal propagators]
\label{lem:ac3-near}
For $\psi\in W^{k+2,r}(\Omega)$ and every $\sigma\in\mathfrak S_3$,
\begin{equation}\label{eq:ac3-near}
\|(E^\sigma(\tau)-I)\psi\|_{k,r}
\leq C\tau\|\psi\|_{k+2,r}.
\end{equation}
\end{lemma}

\begin{proof}
The semigroup integral identities and \eqref{eq:ac3-linear} give
\[
\|(S_{\cA}(\tau)-I)\psi\|_{k,r}
\leq C\tau\|\psi\|_{k+1,r},
\qquad
\|(S_{\cL}(\tau)-I)\psi\|_{k,r}
\leq C\tau\|\psi\|_{k+2,r}.
\]
Indeed,
\[
S_{\cA}(\tau)\psi-\psi
=
\int_0^\tau \cA S_{\cA}(s)\psi\,\mathrm ds,
\qquad
S_{\cL}(\tau)\psi-\psi
=
\int_0^\tau \cL S_{\cL}(s)\psi\,\mathrm ds,
\]
and the corresponding stability estimates hold also for $r=\infty$.

For either ordering $E^\sigma=S_iS_j$, we use
\[
E^\sigma-I
=
S_i(S_j-I)+(S_i-I).
\]
The stability of the subflows on $W^{k,r}$ then yields
\eqref{eq:ac3-near}.
\end{proof}
% \begin{lemma}[Near-identity estimate for the principal propagators]\label{lem:ac3-near}
% For $\psi\in W^{k+2,r}(\Omega)$ and every $\sigma\in\mathfrak S_3$,
% \begin{equation}\label{eq:ac3-near}
%  \|(E^\sigma(\tau)-I)\psi\|_{k,r}\leq C\tau\|\psi\|_{k+2,r}.
% \end{equation}
% \end{lemma}
% \begin{proof}
% The semigroup integral identities and \eqref{eq:ac3-linear} give
% \[
%  \|(S_{\cA}(\tau)-I)\psi\|_{k,r}\leq C\tau\|\psi\|_{k+1,r},\qquad
%  \|(S_{\cL}(\tau)-I)\psi\|_{k,r}\leq C\tau\|\psi\|_{k+2,r}.
% \]
% For either ordering $E^\sigma=S_iS_j$, use
% \[
%  E^\sigma-I=S_i(S_j-I)+(S_i-I)
% \]
% and the stability of $S_i$ on $X$ to obtain \eqref{eq:ac3-near}. For $r=\infty$, periodic mollification of $\psi$ proves the estimate by convergence in $W^{k,\infty}$ and uniform boundedness in $W^{k+2,\infty}$.
% \end{proof}

\begin{proposition}[Verification of the abstract convergence assumptions]\label{prop:ac3-template}
For \eqref{eq:ac3-pde}, the assumptions of Theorem~\ref{thm:general:abstract_convergence} hold with the spaces in \eqref{eq:ac3-spaces}, the coefficients in \eqref{eq:ac3-bracket}, and
\begin{equation}\label{eq:ac3-En}
 E_n=E^{\sigma^n}(\tau).
\end{equation}
\end{proposition}
\begin{proof}
Proposition~\ref{prop:ac3-bounds} supplies a fixed bounded set containing all required exact, numerical, and intermediate states. Property (i) of Theorem~\ref{thm:general:abstract_convergence} follows from Lemma~\ref{lem:ac3-local}, and property (ii) follows from Lemma~\ref{lem:ac3-propagation} with $a=u_n$ and $b=u(t_n)$.

For property (iii), use $G_{n,q}$ and $g_n^{ij}(t_q)$ as in \eqref{eq:general:propagator}--\eqref{eq:general:g_n}. By \eqref{eq:ac3-E},
\begin{equation}\label{eq:ac3-G}
 \|G_{n,q}\|_{\mathcal L(X)}\leq e^{C(n-q)\tau}\leq e^{CT}.
\end{equation}
The explicit coefficients satisfy
\begin{equation}\label{eq:ac3-coeffreg}
 \sup_{0\leq t\leq T}
 \bigl(\|\Phi_{ij}(u(t))\|_{k+2,r}
 +\|\partial_t\Phi_{ij}(u(t))\|_{k,r}\bigr)\leq C.
\end{equation}
Indeed, $\Phi_{13}=0$. The map $\Phi_{12}$ is a linear differential operator of order two, so
\[
 \partial_t\Phi_{12}(u)=\Phi_{12}(u_t).
\]
The equation and the uniform $Y$ bound control $u_t$ in $W^{k+4,r}\cap W^{k+3,\infty}$, which gives the required estimates. For $\Phi_{23}$,
\[
 \partial_t\Phi_{23}(u)
 =-6\nu\bigl(u_t|\nabla u|^2+2u\nabla u\cdot\nabla u_t\bigr),
\]
and Leibniz' rule gives the same bounds. Equations \eqref{eq:ac3-G}--\eqref{eq:ac3-coeffreg} imply $\|g_n^{ij}(t_q)\|_{k,r}\leq C$.

For $0\leq q<n$, the identity $G_{n,q}=G_{n,q+1}E_{q+1}$ yields
\begin{align}\label{eq:ac3-increment}
 g_n^{ij}(t_{q+1})-g_n^{ij}(t_q)
 =G_{n,q+1}\bigl[\Phi_{ij}(u(t_{q+1}))-E_{q+1}\Phi_{ij}(u(t_q))\bigr].
\end{align}
Split the bracket into
\[
 \Phi_{ij}(u(t_{q+1}))-\Phi_{ij}(u(t_q))
 +(I-E_{q+1})\Phi_{ij}(u(t_q)).
\]
The first term has $X$-norm at most $C\tau$ by \eqref{eq:ac3-coeffreg}. The second has the same bound by Lemma~\ref{lem:ac3-near} and the $W^{k+2,r}$ estimate in \eqref{eq:ac3-coeffreg}. Therefore
\[
 \|g_n^{ij}(t_{q+1})-g_n^{ij}(t_q)\|_{k,r}\leq C\tau,
\]
which proves property (iii).
\end{proof}

\begin{theorem}[Three-operator Allen--Cahn convergence]\label{thm:ac3-convergence}
Let $T>0$, $k\in\mathbb N$, and $r\in[2,\infty]$. Assume \eqref{eq:ac3-velocity} and
\[
 u_0\in W^{k+6,r}(\Omega)\cap W^{k+5,\infty}(\Omega).
\]
Let $u(t)$ solve \eqref{eq:ac3-pde}, and let $u_n$ be generated by the
three-operator quasi-random splitting scheme, with the orderings generated
by \eqref{eq:setup:interval_permutation_rule} from the radical-inverse
sequence in Theorem~\ref{thm:prelim:pairwise_cancellation}. Then there exist
$N_0\geq2$ and $C>0$, independent of $N$, such that, for all $N\geq N_0$,
\begin{equation}\label{eq:ac3-convergence}
 \max_{0\leq n\leq N}\|u_n-u(t_n)\|_{k,r}
 \leq C\tau^2(1+\log N).
\end{equation}
For fixed terminal time, this is $O(\tau^2(1+|\log\tau|))$.
\end{theorem}
\begin{proof}
For smooth initial data, Proposition~\ref{prop:ac3-template} verifies all
assumptions of Theorem~\ref{thm:general:abstract_convergence}. The pairwise
signs satisfy the logarithmic cancellation estimate of
Theorem~\ref{thm:prelim:pairwise_cancellation}. Hence
Theorem~\ref{thm:general:abstract_convergence} gives
\eqref{eq:ac3-convergence}.

For general $u_0$ in the stated space, take periodic mollifications $u_0^\varepsilon$. Their $Y$ norms are uniformly bounded, and they converge to $u_0$ in $X$, including $r=\infty$, because of the additional bounded spatial derivatives. The exact difference of two solutions satisfies an advection--diffusion equation with multiplication coefficient $1-u^2-u\widetilde u-\widetilde u^2$, uniformly bounded in $W^{k,\infty}$. The differentiated energy estimate, or the maximum-principle estimate for $r=\infty$, gives stability in $X$. Lemma~\ref{lem:ac3-propagation} gives the corresponding continuity of the splitting steps. Thus, for each fixed $N$, both solutions converge in $X$ as $\varepsilon\to0$. Since the constants in \eqref{eq:ac3-convergence} are uniform in $\varepsilon$, the estimate passes to the limit.
\end{proof}

\section{Numerical experiments}\label{sec:numerics}

In this section we present numerical experiments in bounded-linear and Allen--Cahn-type settings to examine the convergence behavior of quasi-random operator orderings. Once the low-discrepancy driver is fixed, the quasi-random computation is deterministic, so only one quasi-random trajectory is used for each time-step size. For the randomized baseline, repeated independent realizations are used to approximate the expected single-run error. Following the convention in \cite{li2025convergence}, the error norms are first averaged over the random realizations at each time level, and the maximum over time is then taken.

\subsection{Experiments for bounded linear systems}

% We first consider the linear ODE model
% \begin{equation}\label{eq:num:ode_model}
% \frac{\mathrm d}{\mathrm dt}u(t)=\sum_{j=1}^{p}A_j u(t),
% \qquad
% u(0)=u_0,
% \end{equation}
% where $A_1,\dots,A_p\in\mathbb R^{m\times m}$ are noncommuting matrices. 
% Although the convergence analysis in Section~\ref{sec:ode} is carried out for the two-operator case, Algorithm~\ref{Alg:qmc} applies without modification to any number of operators $p\ge2$. 
We first consider a finite-dimensional realization of the bounded linear problem,
\[
\frac{d}{dt}u(t)=\sum_{j=1}^p A_j u(t), \qquad u(0)=u_0,
\]
where \(u(t)\in \mathbb{R}^m\) and \(A_1,\dots,A_p \in \mathbb{R}^{m\times m}\) are noncommuting bounded linear operators. The quasi-random solutions are generated by a deterministic low-discrepancy $p$-operator ordering, with each subflow applied exactly once per time step.

% In these bounded-linear experiments we fix $m=20$,
% $p=3$, $T=1$.
For the comparison in Figure~\ref{fig:exoer}, we fix $m=20$, $p=3$, and $T=1$.
The initial vector $u_0\in\mathbb R^m$ is generated randomly and then normalized so that
$\|u_0\|_2=1$.

The matrices $A_1,\dots,A_p$ are generated as follows. 
Each $A_j$ is sampled as a dense random matrix with independent standard Gaussian entries and then rescaled by
\[
A_j\leftarrow \frac{A_j}{\max\{1,\|A_j\|_2\}},
\qquad j=1,\dots,p.
\]
To exclude nearly commuting matrix collections, we reject the sample and redraw the matrices whenever
\[
\max_{1\le i<j\le p}\|[A_i,A_j]\|_F<10^{-8}.
\]
The exact solution is given by
$u(t)=\exp\!\Bigl(t\sum_{j=1}^{p}A_j\Bigr)u_0$.
The matrix exponential here and in each split subflow is evaluated by a standard matrix exponential routine.

For the quasi-random splitting method, the operator ordering varies
deterministically from step to step according to a fixed low-discrepancy
driver. For the randomized baseline, the operator ordering is sampled
independently and uniformly from $\mathfrak S_p$ at each time step.

We test the time steps
\begin{equation}\label{eq:num:ode_steps}
\tau=2^{-q},
\qquad q=4,5,6,7,8.
\end{equation}
For a given time step $\tau$, let $N=T/\tau$.

For the quasi-random splitting method, let
$\{u_n^{\mathrm{qr}}\}$ be the numerical solution, and define the error sequence
\[
\varepsilon_n^{\mathrm{qr}}:=u_n^{\mathrm{qr}}-u(t_n).
\]
The global-in-time error is measured by
\begin{equation}\label{eq:num:ode_qr_error}
\mathcal E_{\mathrm{qr}}^{\mathrm{lin}}(\tau)
:=
\max_{0\le n\le N}\|\varepsilon_n^{\mathrm{qr}}\|_2.
\end{equation}

For the randomized splitting method, we perform $N_E=10^3$ independent runs. 
Let
\[
u_n^{(\ell)},
\qquad
\ell=1,\dots,N_E,
\qquad
n=0,1,\dots,N,
\]
denote the numerical solution from the $\ell$-th run, and define
\[
\varepsilon_n^{(\ell)}:=u_n^{(\ell)}-u(t_n).
\]
The corresponding empirical mean error is defined by
\begin{equation}\label{eq:num:ode_rand_error}
\mathcal E_{\mathrm{rand}}^{\mathrm{lin}}(\tau)
:=
\max_{0\le n\le N}
\frac{1}{N_E}\sum_{\ell=1}^{N_E}\|\varepsilon_n^{(\ell)}\|_2.
\end{equation}

In the bounded-linear figures, we compare the deterministic quasi-random global-in-time error $\mathcal E_{\mathrm{qr}}^{\mathrm{lin}}(\tau)$ with the maximum-in-time empirical mean error $\mathcal E_{\mathrm{rand}}^{\mathrm{lin}}(\tau)$ of the randomized splitting method. This statistic approximates $\max_n\mathbb E\|u_n^{\rm rand}-u(t_n)\|_2$; it is not an ensemble-average estimator error. 
All curves are plotted against $\tau$ in log--log scale, together with reference lines of slopes $1$, $1.5$, and $2$.

\begin{figure}[!ht]
\centering
\includegraphics[width=\textwidth]{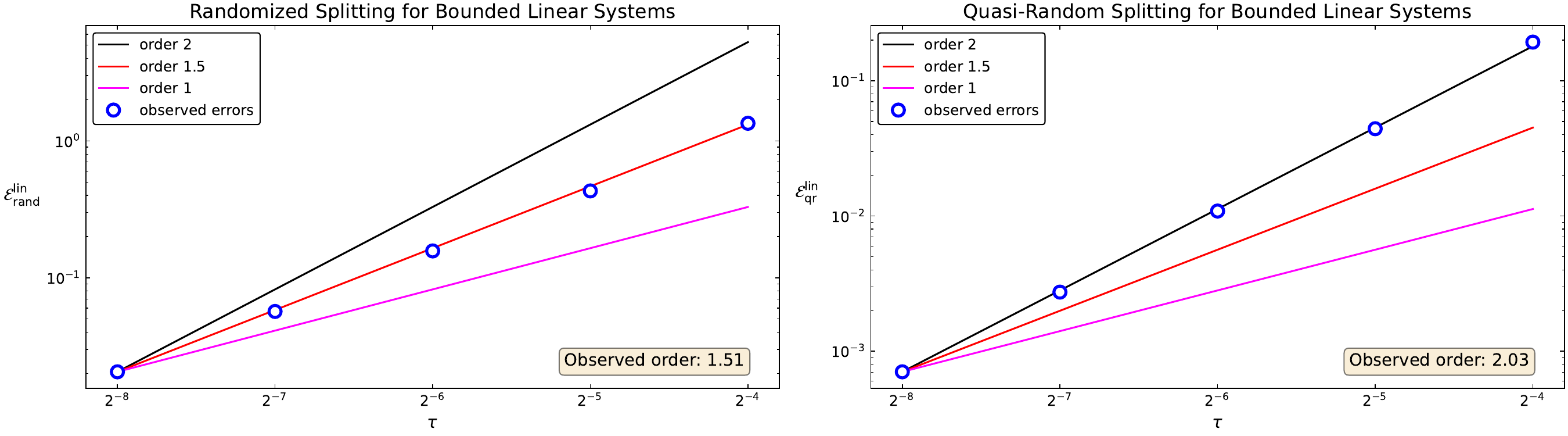}
\caption{
Convergence comparison for the bounded linear system with $p=3$ in the
Euclidean norm. The left and right panels show the randomized and
quasi-random splitting methods, respectively, for
$\tau=2^{-4},2^{-5},\dots,2^{-8}$.
}
\label{fig:exoer}
\end{figure}

The rate plots above are intended to verify the predicted cancellation and convergence behavior. In the multi-operator setting $p\ge3$, every quasi-random time step applies each subflow exactly once and therefore uses $p$ subflow applications, while the standard symmetric multi-operator Strang splitting uses asymptotically $2p-2$ subflow applications per time step.

We also performed an additional supplementary test for the quasi-random splitting method with multiple operators, taking
\[
p=2,3,4,5,6,7.
\]
The corresponding errors are shown in Figure~\ref{fig:ode_multi_p_qr}.
The corresponding results indicate that the near-second-order behavior persists numerically for several larger values of $p$, consistent with the essentially second-order fixed-$p$ behavior established in the multi-operator analysis.

\begin{figure}[!ht]
\centering
\includegraphics[width=\textwidth]{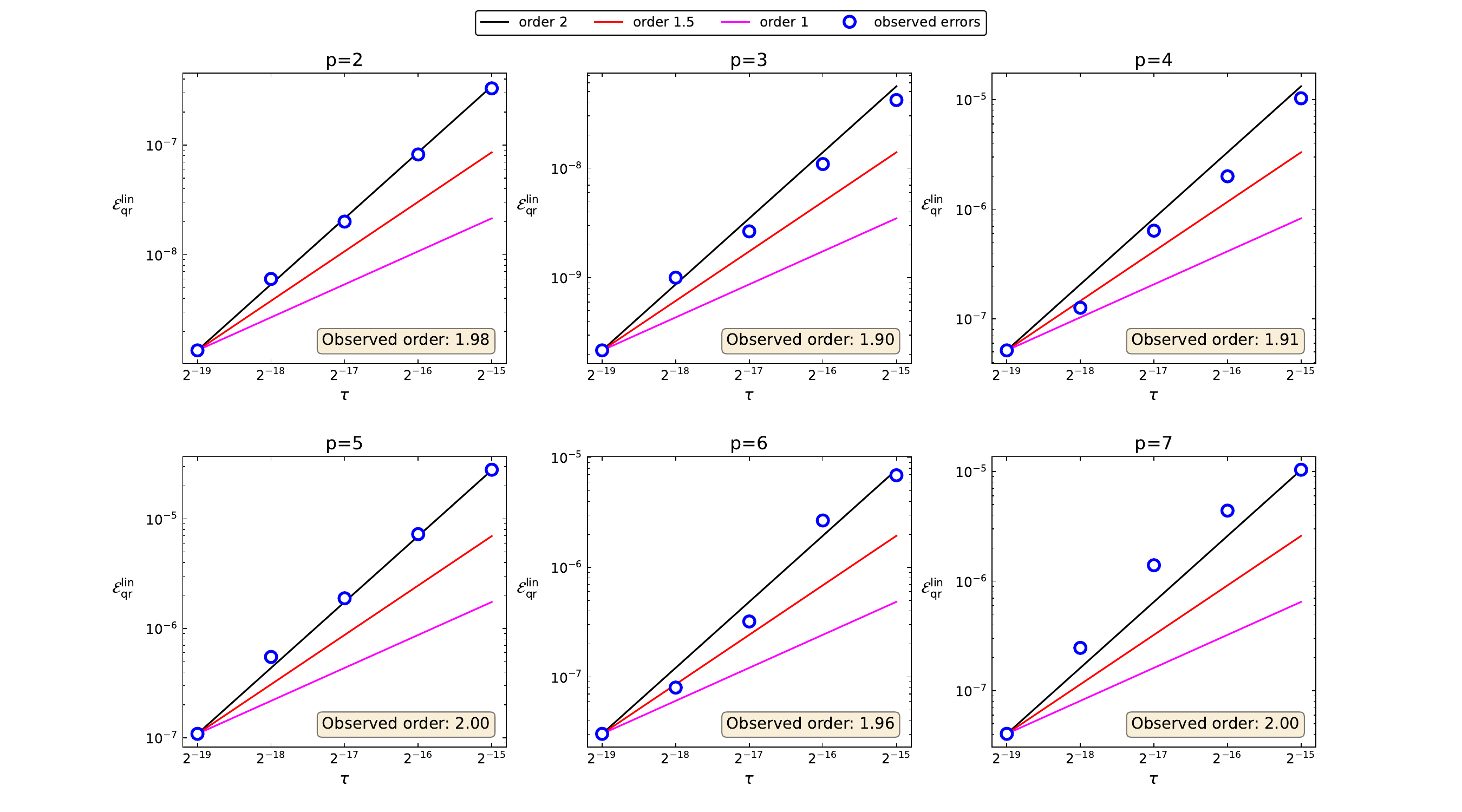}
\caption{
Convergence of the quasi-random splitting method for bounded linear
systems with $p=2,3,4,5,6,7$. The errors are measured in the Euclidean
norm for $\tau=2^{-15},2^{-16},\dots,2^{-19}$. The observed rates remain
close to two for all tested values of $p$.
}
\label{fig:ode_multi_p_qr}
\end{figure}

\subsection{Experiments for the Allen--Cahn equation}

In this subsection, we present two sets of experiments for the Allen--Cahn
equation. We first consider the two-operator diffusion--reaction problem as
a simpler test for comparing the quasi-random and randomized splitting
methods. We then consider the three-operator Allen--Cahn equation with a
background flow, which corresponds directly to the nonlinear PDE problem
analyzed in Section~\ref{sec:pde}. The latter experiment is used to examine
the essentially second-order convergence predicted by
Theorem~\ref{thm:ac3-convergence}.

\subsubsection{Two-operator diffusion--reaction test}

We first consider the Allen--Cahn equation
\begin{equation}\label{eq:num:pde_model}
\partial_t u=\cL u+\cR(u)=\nu\Delta u-f(u),
\qquad
f(u)=u^3-u,
\end{equation}
posed on the two-dimensional torus
\[
\Omega=\mathbb T^2=[0,L]^2
\]
with periodic boundary conditions.
Following the spatial-temporal setting used in the reference numerical study, but adapted to the present two-operator Allen--Cahn model, we take
\begin{equation}\label{eq:num:pde_parameters}
L=2\pi,
\qquad
\nu=1,
\qquad
T=1,
\end{equation}
and the initial value
\begin{equation}\label{eq:num:pde_initial}
u_0(x,y)=1+0.5\sin x+\exp(0.7\sin y).
\end{equation}

The spatial discretization is carried out by the Fourier spectral method.
More precisely, we use the uniform mesh
\begin{equation}\label{eq:num:pde_mesh}
\Delta x=\Delta y=h:=\pi\,2^{-7},
\end{equation}
and all spatial derivatives are computed spectrally through the discrete Fourier transform.
This choice makes the spatial error negligible compared with the temporal splitting error on the range of time steps considered below.

For the two subproblems in the splitting method, the linear heat flow generated by $\cL$ is solved exactly in Fourier space:
\[
\widehat{S_{\cL}(\tau)w}(k)
=
\exp\!\Bigl(-4\pi^2\nu |k|^2\tau/L^2\Bigr)\hat w(k).
\]
The reaction flow generated by $\cR$ is solved pointwise in physical space by the explicit formula
\begin{equation}\label{eq:num:pde_reaction}
S_{\cR}(\tau)[w]
=
\frac{w}{\sqrt{w^2+(1-w^2)e^{-2\tau}}}.
\end{equation}
Under the chosen spatial discretization, both subflows are advanced exactly in time. Thus the dominant temporal error of the splitting computation is the splitting error; the numerical reference and the spatial discretization still contribute smaller residual errors.

The numerical reference solution, denoted by $u^{\rm ref}$, is computed with the same Fourier spectral discretization in space and a second-order exponential midpoint rule in time. Let
\[
\lambda_k:=\frac{4\pi^2\nu}{L^2}|k|^2,
\qquad
\varphi_1(z):=\begin{cases}(e^z-1)/z,&z\ne0,\\1,&z=0.\end{cases}
\]
Writing $h_{\rm ref}$ for the reference time step and using $\cR(u)=u-u^3$, one reference step is
\begin{align}
\widehat{u}^{\,m+1/2}(k)
&=
 e^{-\lambda_k h_{\rm ref}/2}\widehat{u}^{\,m}(k)
 +\frac{h_{\rm ref}}{2}\,
 \varphi_1(-\lambda_k h_{\rm ref}/2)\,
 \widehat{\cR(u^m)}(k),
 \label{eq:num:pde_ref_half}\\
\widehat{u}^{\,m+1}(k)
&=
 e^{-\lambda_k h_{\rm ref}}\widehat{u}^{\,m}(k)
 +h_{\rm ref}\,
 \varphi_1(-\lambda_k h_{\rm ref})\,
 \widehat{\cR(u^{m+1/2})}(k).
 \label{eq:num:pde_ref_full}
\end{align}
For the reported experiments we take
\begin{equation}\label{eq:num:pde_reference_step}
h_{\rm ref}=\tau_{\rm ref}=2^{-20}.
\end{equation}
Thus the reference temporal discretization is substantially finer than every splitting time step used below.
For the splitting schemes, we test the time steps
\begin{equation}\label{eq:num:pde_steps}
\tau=2^{-m},
\qquad m=10,11,...,15.
\end{equation}

Since the quasi-random ordering is deterministic once the low-discrepancy sequence is fixed, only one quasi-random trajectory is used for each value of $\tau$. 

Let $u^{{\rm qr}}_{n,ij}$ be the numerical solution produced by the quasi-random splitting method at the grid point $(x_i,y_j)=(ih,jh)$ and time $t_n=n\tau$, and let
\[
e^{{\rm qr}}_{n,ij}:=u^{{\rm qr}}_{n,ij}-u^{\rm ref}_{n,ij}
\]
be the corresponding pointwise error.
We measure the error in the discrete $L^2$ norm and the discrete $W^{1,2}$ norm.
More precisely, we define
\begin{equation}\label{eq:num:pde_l2_error}
\mathcal E^{h}_{2,{\rm qr}}(\tau)
:=
\max_{0\le n\le N}
\left(
\sum_{i,j}|e^{{\rm qr}}_{n,ij}|^2\,\Delta x\Delta y
\right)^{1/2},
\end{equation}
and
\begin{equation}\label{eq:num:pde_h1_error}
\mathcal E^{h}_{1,2,{\rm qr}}(\tau)
:=
\max_{0\le n\le N}
\left(
\sum_{i,j}
\Bigl(
|e^{{\rm qr}}_{n,ij}|^2
+
|\nabla_h e^{{\rm qr}}_{n,ij}|^2
\Bigr)
\,\Delta x\Delta y
\right)^{1/2}.
\end{equation}
where $\nabla_h$ denotes the Fourier spectral gradient. The same discrete norms are used for the randomized splitting method.

For the randomized splitting method, let $u^{(\ell)}_{n,ij}$ denote the $\ell$-th realization and define
\[
e^{(\ell)}_{n,ij}:=u^{(\ell)}_{n,ij}-u^{\rm ref}_{n,ij}.
\]
Following \cite{li2025convergence}, we repeat the randomized splitting experiment independently $N_E=10^3$ times. At each time level, the expected single-run error is approximated by the empirical mean over these independent realizations, after which the maximum over the whole time interval is taken.
\begin{equation}\label{eq:num:pde_rand_l2_error}
\mathcal E^{h}_{2,{\rm rand}}(\tau)
:=
\max_{0\le n\le N}
\frac{1}{N_E}\sum_{\ell=1}^{N_E}
\left(
\sum_{i,j}|e^{(\ell)}_{n,ij}|^2\,\Delta x\Delta y
\right)^{1/2},
\end{equation}
and
\begin{equation}\label{eq:num:pde_rand_h1_error}
\mathcal E^{h}_{1,2,{\rm rand}}(\tau)
:=
\max_{0\le n\le N}
\frac{1}{N_E}\sum_{\ell=1}^{N_E}
\left(
\sum_{i,j}\Bigl(|e^{(\ell)}_{n,ij}|^2+|\nabla_h e^{(\ell)}_{n,ij}|^2\Bigr)\,\Delta x\Delta y
\right)^{1/2}.
\end{equation}
% For the randomized splitting method, let $u^{(\ell)}_{n,ij}$ denote the $\ell$-th realization and define
% \[
% e^{(\ell)}_{n,ij}:=u^{(\ell)}_{n,ij}-u^{\rm ref}_{n,ij}.
% \]
% We repeat the experiment independently $N_E=10^3$ times and report the empirical mean of the global-in-time pathwise errors
% \begin{equation}\label{eq:num:pde_rand_l2_error}
% \mathcal E^{h}_{2,{\rm rand}}(\tau)
% :=
% \frac{1}{N_E}\sum_{\ell=1}^{N_E}
% \max_{0\le n\le N}
% \left(
% \sum_{i,j}|e^{(\ell)}_{n,ij}|^2\,\Delta x\Delta y
% \right)^{1/2},
% \end{equation}
% and
% \begin{equation}\label{eq:num:pde_rand_h1_error}
% \mathcal E^{h}_{1,2,{\rm rand}}(\tau)
% :=
% \frac{1}{N_E}\sum_{\ell=1}^{N_E}
% \max_{0\le n\le N}
% \left(
% \sum_{i,j}\Bigl(|e^{(\ell)}_{n,ij}|^2+|\nabla_h e^{(\ell)}_{n,ij}|^2\Bigr)\,\Delta x\Delta y
% \right)^{1/2}.
% \end{equation}

Figure~\ref{fig:pde_two_l2} and Figure~\ref{fig:pde_two_h12} display the convergence behavior of the two-operator diffusion--reaction test in the discrete $L^2$ and $W^{1,2}$ norms, respectively.
The quasi-random splitting method is deterministic once the underlying low-discrepancy driver is fixed, so only one trajectory is used for each value of $\tau$. For the randomized baseline, the expected single-run error is approximated by averaging the error norms over independent realizations at each time level. The results show an essentially second-order convergence behavior for the quasi-random method, while the randomized baseline exhibits a slower convergence rate close to $3/2$.

\begin{figure}[!ht]
\centering
\includegraphics[width=\textwidth]{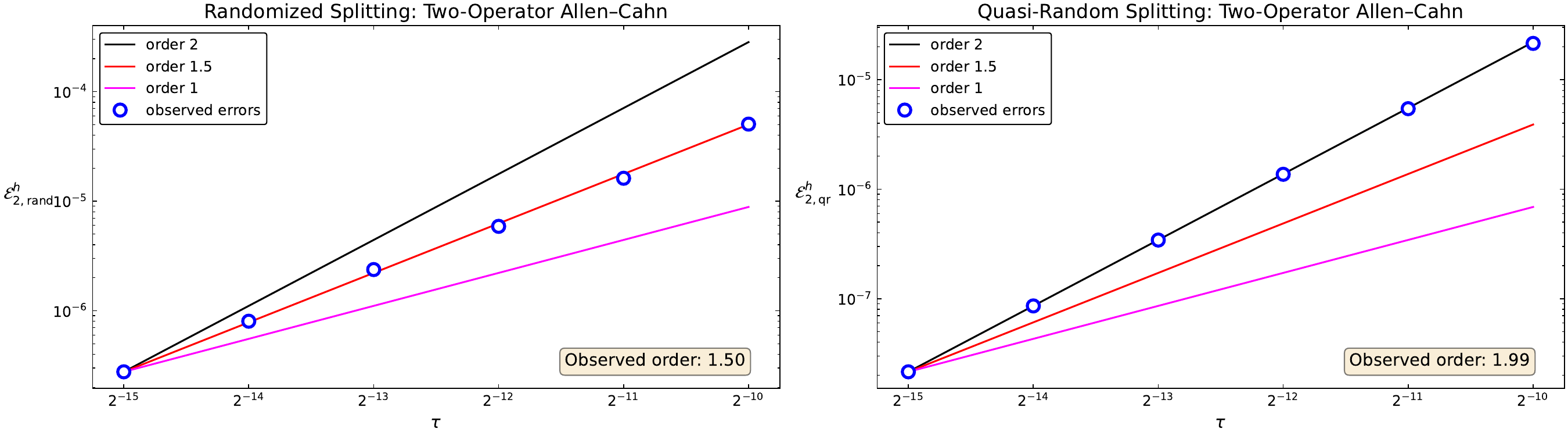}
\caption{
Convergence comparison for the two-operator diffusion--reaction
Allen--Cahn equation in the discrete $L^2$ norm. The left and right
panels show the randomized and quasi-random splitting methods,
respectively, for
$\tau=2^{-10},2^{-11},\dots,2^{-15}$.
}
\label{fig:pde_two_l2}
\end{figure}

\begin{figure}[!ht]
\centering
\includegraphics[width=\textwidth]{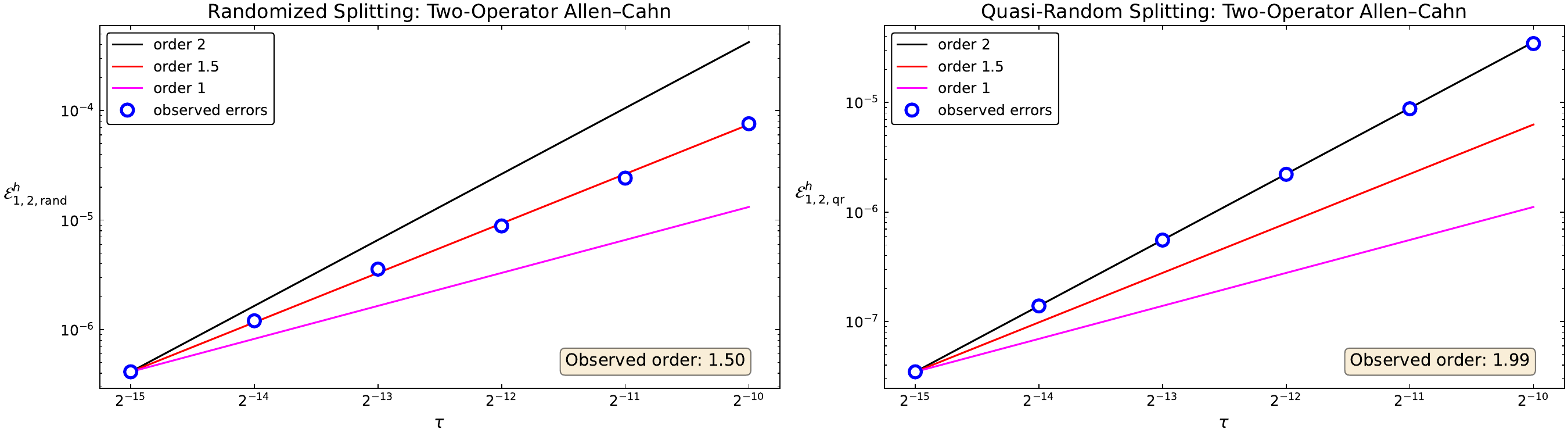}
\caption{
Convergence comparison for the two-operator diffusion--reaction
Allen--Cahn equation in the discrete $W^{1,2}$ norm. The left and right
panels show the randomized and quasi-random splitting methods,
respectively, for
$\tau=2^{-10},2^{-11},\dots,2^{-15}$.
}
\label{fig:pde_two_h12}
\end{figure}

\subsubsection{Three-operator test with background flow}

We next consider the Allen--Cahn equation with a background flow
\begin{equation}\label{eq:num:pde_model_three}
\partial_t u + v(x)\cdot \nabla u = \nu \Delta u - f(u),
\qquad
f(u)=u^3-u.
\end{equation}
We keep exactly the same domain, parameters, initial value, spatial discretization, and reference time step as in
\eqref{eq:num:pde_parameters}--\eqref{eq:num:pde_reference_step}.
Thus the only new ingredient is the advection operator
\[
\cA u:=-v(x)\cdot\nabla u,
\]
for which we take the same shear flow as in the reference three-operator experiment,
\begin{equation}\label{eq:num:pde_velocity}
v(x,y)=(-0.75\sin y,\,0).
\end{equation}

The three subflows are then given by advection, diffusion, and reaction.
The diffusion step is again solved exactly in Fourier space, and the reaction step is still given by \eqref{eq:num:pde_reaction}.
For the advection step, we advance the transport equation
\[
\partial_t w + v(x)\cdot \nabla w = 0
\]
by the classical fourth-order Runge--Kutta method in physical space.
For the numerical reference solution, denoted by $u^{\rm ref,3}$, we use the same second-order exponential midpoint construction \eqref{eq:num:pde_ref_half}--\eqref{eq:num:pde_ref_full}, now with the non-diffusive term replaced by
\[
\mathcal G(u):=\cA u+\cR(u)=-v\cdot\nabla u+u-u^3.
\]
The reference time step remains $\tau_{\rm ref}=2^{-20}$. The RK4 discretization of the advection subflow is fourth order in time and is chosen sufficiently accurate that it does not affect the observed near-second-order splitting slope over the reported range.

We use the same family of time steps as in \eqref{eq:num:pde_steps}.
Since the quasi-random splitting method is deterministic once the low-discrepancy driver is fixed, we report only the single-run errors in the two norms.
Let $u^{{\rm qr},3}_{n,ij}$ be the three-operator quasi-random splitting solution and define
\[
e^{{\rm qr},3}_{n,ij}:=u^{{\rm qr},3}_{n,ij}-u^{{\rm ref},3}_{n,ij}.
\]
The corresponding discrete $L^2$ and $W^{1,2}$ errors are defined by
\begin{equation}\label{eq:num:pde_three_l2_error}
\mathcal E^{h,3}_{2,{\rm qr}}(\tau)
:=
\max_{0\le n\le N}
\left(
\sum_{i,j}|e^{{\rm qr},3}_{n,ij}|^2\,\Delta x\Delta y
\right)^{1/2},
\end{equation}
and
\begin{equation}\label{eq:num:pde_three_h1_error}
\mathcal E^{h,3}_{1,2,{\rm qr}}(\tau)
:=
\max_{0\le n\le N}
\left(
\sum_{i,j}
\Bigl(
|e^{{\rm qr},3}_{n,ij}|^2
+
|\nabla_h e^{{\rm qr},3}_{n,ij}|^2
\Bigr)
\,\Delta x\Delta y
\right)^{1/2}.
\end{equation}

Figure~\ref{fig:pde_three} shows the convergence behavior of the
quasi-random splitting method for the three-operator problem in the discrete
$L^2$ and $W^{1,2}$ norms. In both norms, the observed convergence rate is
close to two, consistent with the essentially second-order estimate in
Theorem~\ref{thm:ac3-convergence}.

\begin{figure}[!ht]
\centering
\includegraphics[width=\textwidth]{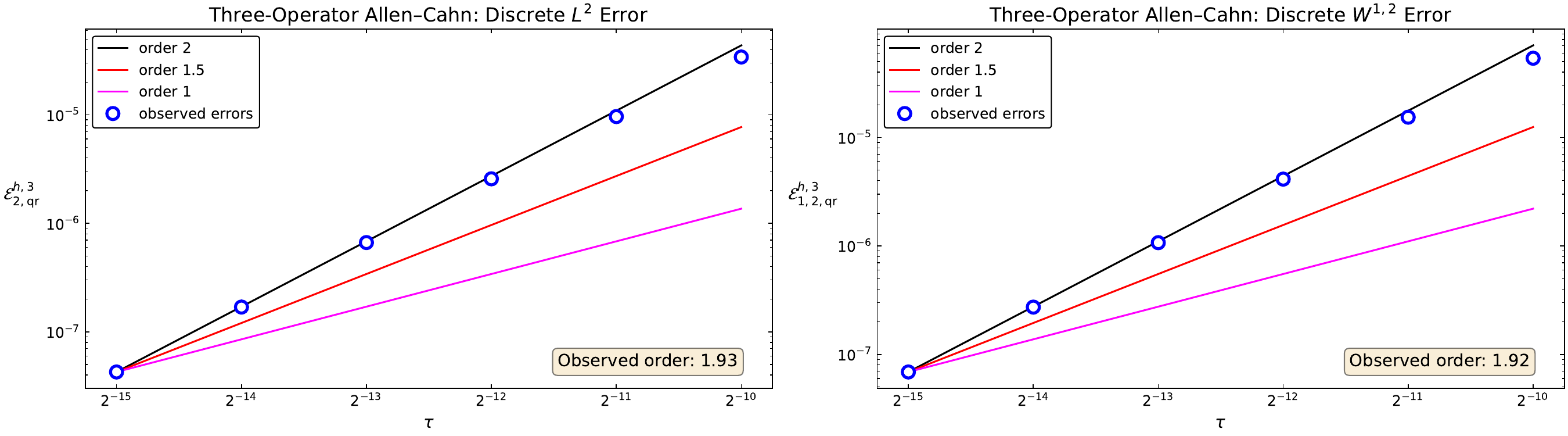}
\caption{Convergence of the three-operator quasi-random splitting method
for the Allen--Cahn equation with a background flow. The left and right
panels show the errors in the discrete $L^2$ and $W^{1,2}$ norms,
respectively, for
$\tau=2^{-10},2^{-11},\dots,2^{-15}$.}
\label{fig:pde_three}
\end{figure}

\section{Conclusion}\label{sec:conclusion}

We developed a quasi-random operator splitting method based on a
one-dimensional point-to-permutation construction. At each time step, a
low-discrepancy point is decoded directly into a permutation of the $p$
split operators, and every subflow is applied exactly once. The method
therefore retains the Lie-type cost of $p$ subflow applications per step;
for $p\ge3$, this is lower than the asymptotic $2p-2$ subflow applications
of the standard symmetric multi-operator Strang splitting.

For bounded linear problems with an arbitrary fixed number $p$ of operators,
we showed that the leading $O(\tau^2)$ term in the local splitting defect is
a sum of pairwise commutator contributions whose signs are determined by
the ordering of the corresponding operator pairs. Through the equivalent
interval representation of the point-to-permutation map, each pairwise sign
can be expressed as a one-dimensional bounded-variation function of the
low-discrepancy point. The discrepancy estimate for the radical-inverse
sequence then yields logarithmic cancellation of every pairwise sign
sequence. Combining this cancellation with stability, propagation
estimates, and Abel summation gives
\[
\max_{0\le n\le N}\|u_n-u(t_n)\|
\le
C_p\tau^2(1+\log N),
\]
which is essentially second order for every fixed $p$.

For the nonlinear three-operator Allen--Cahn equation with a background
flow, we established an abstract convergence theorem for three-operator
splitting based on three key properties: a pairwise expansion of the
local splitting defect, stability of one-step error propagation, and
temporal regularity of the propagated defect coefficients. Under suitable
Sobolev regularity assumptions, these properties yield the same
essentially second-order estimate,
\[
\max_{0\le n\le N}\|u_n-u(t_n)\|_{k,r}
\le
C\tau^2(1+\log N).
\]
The numerical experiments for bounded linear systems and the two- and
three-operator Allen--Cahn problems exhibit convergence behavior consistent
with the analysis and illustrate that the quasi-random strategy can improve
substantially on first-order Lie splitting while retaining one application
of each subflow per time step.

Future work includes extending the present nonlinear analysis to general
$p$-operator nonlinear and multiphysics PDEs, improving the dependence of
the convergence constants on the number of operators, and designing
permutation enumerations with smaller pairwise variation constants.

\subsection*{Acknowledgement}
This work is partially supported by NSFC 12371400 and Shanghai Municipal Science and Technology Major Project 2021SHZDZX0102.

\bibliographystyle{unsrt}
\bibliography{references}

@book{hairer2006geometric,
  title={Geometric Numerical Integration},
  author={Ernst Hairer and Gerhard Wanner and Christian Lubich},
  year={2006},
  publisher={Springer}
}

@article{praetorius2015navier,
  title={A {Navier-Stokes} phase-field crystal model for colloidal suspensions},
  author={Praetorius, Simon and Voigt, Axel},
  journal={The Journal of chemical physics},
  volume={142},
  number={15},
  year={2015},
  publisher={AIP Publishing}
}

@article{jacqmin1999calculation,
  title={Calculation of two-phase {Navier--Stokes} flows using phase-field modeling},
  author={Jacqmin, David},
  journal={Journal of computational physics},
  volume={155},
  number={1},
  pages={96--127},
  year={1999},
  publisher={Elsevier}
}

@article{strang1968construction,
  title={On the construction and comparison of difference schemes},
  author={Strang, Gilbert},
  journal={SIAM journal on numerical analysis},
  volume={5},
  number={3},
  pages={506--517},
  year={1968},
  publisher={SIAM}
}

@book{iserles2009first,
  title={A first course in the numerical analysis of differential equations},
  author={Iserles, Arieh},
  number={44},
  year={2009},
  publisher={Cambridge university press}
}

@book{glowinski2017splitting,
  title={Splitting methods in communication, imaging, science, and engineering},
  author={Glowinski, Roland and Osher, Stanley J and Yin, Wotao},
  year={2017},
  publisher={Springer}
}

@article{yue2004diffuse,
  title={A diffuse-interface method for simulating two-phase flows of complex fluids},
  author={Yue, Pengtao and Feng, James J and Liu, Chun and Shen, Jie},
  journal={Journal of Fluid Mechanics},
  volume={515},
  pages={293--317},
  year={2004},
  publisher={Cambridge University Press}
}

@article{anderson1998diffuse,
  title={Diffuse-interface methods in fluid mechanics},
  author={Anderson, Daniel M and McFadden, Geoffrey B and Wheeler, Adam A},
  journal={Annual review of fluid mechanics},
  volume={30},
  number={1},
  pages={139--165},
  year={1998},
  publisher={Annual Reviews 4139 El Camino Way, PO Box 10139, Palo Alto, CA 94303-0139, USA}
}

@article{halton1960efficiency,
  title={On the efficiency of certain quasi-random sequences of points in evaluating multi-dimensional integrals},
  author={Halton, John H},
  journal={Numerische Mathematik},
  volume={2},
  number={1},
  pages={84--90},
  year={1960},
  publisher={Springer}
}

@book{kuipers2012uniform,
  title={Uniform distribution of sequences},
  author={Kuipers, Lauwerens and Niederreiter, Harald},
  year={2012},
  publisher={Courier Corporation}
}

@article{li2025convergence,
  title={Convergence of Random Splitting Method for the Allen--Cahn Equation in a Background Flow},
  author={Li, Lei and Wang, Chen},
  journal={Numerical Methods for Partial Differential Equations},
  volume={41},
  number={6},
  pages={e70040},
  year={2025},
  publisher={Wiley Online Library}
}

@article{caflisch1998monte,
  title={Monte carlo and quasi-monte carlo methods},
  author={Caflisch, Russel E},
  journal={Acta numerica},
  volume={7},
  pages={1--49},
  year={1998},
  publisher={Cambridge University Press}
}

@article{li2025ergodicity,
  title={Ergodicity and error estimate of laws for a random splitting Langevin Monte Carlo},
  author={Li, Lei and Wang, Chen and Wang, Mengchao},
  journal={arXiv preprint arXiv:2510.07676},
  year={2025}
}

@article{cho2024doubling,
  title={Doubling the order of approximation via the randomized product formula},
  author={Cho, Chien-Hung and Berry, Dominic W and Hsieh, Min-Hsiu},
  journal={Physical Review A},
  volume={109},
  number={6},
  pages={062431},
  year={2024},
  publisher={APS}
}

@article{mclachlan2002splitting,
  title={Splitting methods},
  author={McLachlan, Robert I and Quispel, G Reinout W},
  journal={Acta Numerica},
  volume={11},
  pages={341--434},
  year={2002},
  publisher={Cambridge University Press}
}

@article{blanes2008splitting,
author = {Blanes, Sergio and Casas, Fernando and Murua, Ander},
year = {2009},
month = {01},
pages = {},
title = {Splitting and composition methods in the numerical integration of differential equations},
volume = {45},
journal = {Bol. Soc. Esp. Mat. Apl.}
}

@article{Blanes_Casas_Murua_2024,
title={Splitting methods for differential equations},
volume={33},
DOI={10.1017/S0962492923000077},
journal={Acta Numerica},
author={Blanes, Sergio and Casas, Fernando and Murua, Ander}, year={2024}, 
pages={1--161}}

@book{dick2010digital,
  title={Digital nets and sequences: discrepancy theory and quasi--Monte Carlo integration},
  author={Dick, Josef and Pillichshammer, Friedrich},
  year={2010},
  publisher={Cambridge University Press}
}

@article{dick2013high,
  title={High-dimensional integration: the quasi-Monte Carlo way},
  author={Dick, Josef and Kuo, Frances Y and Sloan, Ian H},
  journal={Acta Numerica},
  volume={22},
  pages={133--288},
  year={2013},
  publisher={Cambridge University Press}
}

\appendix

\section{Explicit local defect expansion for the two-operator Allen--Cahn splitting}
\label{sec:appendix}

For completeness, we record the explicit local defect expansion for the
two-operator diffusion--reaction Allen--Cahn splitting considered in the
numerical experiments of Section~\ref{sec:numerics}. This calculation is a
special case of the pairwise local defect structure established in
Lemma~\ref{lem:ac3-local}.

Consider
\begin{equation}\label{eq:appendix:model}
\partial_t u
=
\cL u+\cR(u),
\qquad
\cL u:=\nu\Delta u,
\qquad
\cR(u):=u-u^3.
\end{equation}
Let $T(t)$ denote the exact flow of \eqref{eq:appendix:model}, and let
$S_{\cL}(t)$ and $S_{\cR}(t)$ denote the diffusion and reaction subflows,
respectively.

Let $k\in\mathbb N$, $r\in[2,\infty]$, and assume
\[
a\in W^{k+6,r}(\Omega)\cap W^{k+5,\infty}(\Omega).
\]
The diffusion and reaction subflows admit the expansions
\begin{align}
S_{\cL}(\tau)[a]
&=
a+\tau\cL a
+\frac{\tau^2}{2}\cL^2a
+R_{\cL}(\tau,a),
\label{eq:appendix:SL}
\\
S_{\cR}(\tau)[a]
&=
a+\tau\cR(a)
+\frac{\tau^2}{2}
D\cR(a)\cR(a)
+R_{\cR}(\tau,a),
\label{eq:appendix:SR}
\end{align}
where
\begin{equation}\label{eq:appendix:remainder_basic}
\|R_{\cL}(\tau,a)\|_{k,r}
+
\|R_{\cR}(\tau,a)\|_{k,r}
\le
C\tau^3.
\end{equation}
Here and below, $C$ depends only on a bound for the relevant Sobolev norms
of $a$ and is independent of $\tau$.

We distinguish the two possible application orders. Define
\begin{equation}\label{eq:appendix:two_orderings}
S^{LR}(\tau)
:=
S_{\cR}(\tau)S_{\cL}(\tau),
\qquad
S^{RL}(\tau)
:=
S_{\cL}(\tau)S_{\cR}(\tau).
\end{equation}
Thus, $S^{LR}$ applies diffusion first and reaction second, whereas
$S^{RL}$ applies reaction first and diffusion second.

For the ordering $S^{LR}$, substituting
\eqref{eq:appendix:SL} into \eqref{eq:appendix:SR} gives
\begin{align}
S^{LR}(\tau)[a]
&=
a
+\tau\bigl(\cL a+\cR(a)\bigr)
+\frac{\tau^2}{2}
\Bigl(
\cL^2a
+
2D\cR(a)\cL a
+
D\cR(a)\cR(a)
\Bigr)
\notag\\
&\quad
+
R_{LR}(\tau,a),
\label{eq:appendix:SLR}
\end{align}
where
\[
\|R_{LR}(\tau,a)\|_{k,r}
\le
C\tau^3.
\]

For the reverse ordering $S^{RL}$, one obtains
\begin{align}
S^{RL}(\tau)[a]
&=
a
+\tau\bigl(\cL a+\cR(a)\bigr)
+\frac{\tau^2}{2}
\Bigl(
\cL^2a
+
2\cL\cR(a)
+
D\cR(a)\cR(a)
\Bigr)
\notag\\
&\quad
+
R_{RL}(\tau,a),
\label{eq:appendix:SRL}
\end{align}
with
\[
\|R_{RL}(\tau,a)\|_{k,r}
\le
C\tau^3.
\]

The exact flow satisfies
\begin{align}
T(\tau)[a]
&=
a
+\tau\bigl(\cL a+\cR(a)\bigr)
\notag\\
&\quad
+\frac{\tau^2}{2}
\Bigl(
\cL^2a
+
\cL\cR(a)
+
D\cR(a)\cL a
+
D\cR(a)\cR(a)
\Bigr)
+
R_T(\tau,a),
\label{eq:appendix:T}
\end{align}
where
\[
\|R_T(\tau,a)\|_{k,r}
\le
C\tau^3.
\]

The pairwise defect coefficient associated with diffusion and reaction is
\begin{equation}\label{eq:appendix:PhiLR}
\Phi_{LR}(a)
:=
\cL\cR(a)-D\cR(a)\cL a.
\end{equation}
Since
\[
D\cR(a)b=(1-3a^2)b,
\]
we have
\begin{align*}
\Phi_{LR}(a)
&=
\nu\Delta(a-a^3)
-
(1-3a^2)\nu\Delta a.
\end{align*}
Using
\begin{equation}\label{eq:appendix:lap_nonlinear}
\Delta(a-a^3)
=
(1-3a^2)\Delta a
-
6a|\nabla a|^2,
\end{equation}
we obtain
\begin{equation}\label{eq:appendix:PhiLR_explicit}
\Phi_{LR}(a)
=
-6\nu a|\nabla a|^2.
\end{equation}
This is precisely the coefficient $\Phi_{23}$ in
\eqref{eq:ac3-phi23} when the three-operator problem is restricted to the
diffusion--reaction pair.

Subtracting \eqref{eq:appendix:T} from
\eqref{eq:appendix:SLR} gives
\begin{equation}\label{eq:appendix:defect_LR}
S^{LR}(\tau)[a]-T(\tau)[a]
=
-\frac{\tau^2}{2}\Phi_{LR}(a)
+
\widetilde R_{LR}(\tau,a),
\end{equation}
where
\[
\|\widetilde R_{LR}(\tau,a)\|_{k,r}
\le
C\tau^3.
\]
Similarly,
\begin{equation}\label{eq:appendix:defect_RL}
S^{RL}(\tau)[a]-T(\tau)[a]
=
\frac{\tau^2}{2}\Phi_{LR}(a)
+
\widetilde R_{RL}(\tau,a),
\end{equation}
with
\[
\|\widetilde R_{RL}(\tau,a)\|_{k,r}
\le
C\tau^3.
\]

Using \eqref{eq:appendix:PhiLR_explicit}, the two defects can equivalently be
written as
\begin{align}
S^{LR}(\tau)[a]-T(\tau)[a]
&=
3\nu\tau^2 a|\nabla a|^2
+
\widetilde R_{LR}(\tau,a),
\label{eq:appendix:defect_LR_explicit}
\\
S^{RL}(\tau)[a]-T(\tau)[a]
&=
-3\nu\tau^2 a|\nabla a|^2
+
\widetilde R_{RL}(\tau,a).
\label{eq:appendix:defect_RL_explicit}
\end{align}

Hence reversing the order of the diffusion and reaction subflows changes
the sign of the leading $O(\tau^2)$ local defect. This is the two-operator
instance of the pairwise sign mechanism used in the general
multi-operator analysis.

\end{document}